\documentclass[10pt,twocolumn,twoside]{IEEEtran}
\usepackage{generic}
\usepackage{cite}
\usepackage{amsmath,amssymb,amsfonts}
\usepackage{amsthm}
\usepackage{algorithmic}
\usepackage{graphicx}
\usepackage{textcomp}
\usepackage{array}
\usepackage{calc}
\usepackage{subfigure}
\usepackage{color}
\usepackage{enumitem}
\usepackage{multirow}
\usepackage{hhline}
\usepackage{enumitem}
\usepackage[hyphens,spaces]{url}

\def\BibTeX{{\rm B\kern-.05em{\sc i\kern-.025em b}\kern-.08em
    T\kern-.1667em\lower.7ex\hbox{E}\kern-.125emX}}
\newcommand{\bsDelta}{\boldsymbol{{\Delta}}} 
\newcommand{\bsbeta}{\boldsymbol{\beta}}
\newcommand{\bsx}{\mathbf{x}}
\newcommand{\bfx}{\mathbf{x}}
\newcommand{\bfxs}{\mathbf{x}^*}
\newcommand{\bshx}{\what{\mathbf{x}}}
\newcommand{\bsu}{\mathbf{u}} 
\newcommand{\bfu}{\mathbf{u}} 
 
\newcommand{\bfust}{\widetilde{\mathbf{u}}^*} 
\newcommand{\bshu}{\what{\mathbf{u}}}
 
\newcommand{\bshz}{\what{\mathbf{z}}}
\newcommand{\bsy}{\mathbf{y}} 
\newcommand{\bsO}{\mathbf{O}} 
\newcommand{\bsJ}{\mathbf{J}} 
\newcommand{\bsJtil}{\widetilde{\mathbf{J}}}
\newcommand{\bsM}{\mathbf{M}}
\newcommand{\real}{\mathbb{R}}

\newcommand{\mbf}[1]{\mathbf{#1}}
\newcommand{\bs}[1]{\boldsymbol{#1}}
\newcommand{\mcal}[1]{\mathcal{#1}}
\newcommand*{\QEDB}{\hfill\ensuremath{\square}}

\newcommand{\what}{\widehat}
\newcommand{\wtilde}{\widetilde}

\DeclareMathOperator*{\argmin}{arg\,min}

\newcommand{\Rank}{\operatorname{Rank}}

\newcommand{\transpose}{\mathsf{T}}

\newcommand{\bfzeros}{\mathbf{0}}
\newcommand{\vertiii}[1]{{\left\vert\kern-0.25ex\left\vert\kern-0.25ex\left\vert #1 
		\right\vert\kern-0.25ex\right\vert\kern-0.25ex\right\vert}}

\newcommand{\bfA}{\mathbf{A}}
\newcommand{\bfB}{\mathbf{B}}

\newcommand{\bfC}{\mathbf{C}}
\newcommand{\bfJ}{\mathbf{J}}

\newcommand{\sysmatS}{(\bfA, \bfB_S, \bfC)}

\newcommand{\bsPsi}{\boldsymbol{\Psi}}

\newcommand{\norm}[1]{\left\lVert#1\right\rVert}

\newtheorem{theorem}{\bf \emph{Theorem}}
\newtheorem{lemma}[theorem]{\bf \emph{Lemma}}
\newtheorem{definition}{Definition}
\newtheorem{corollary}[theorem]{Corollary}
\newtheorem{proposition}[theorem]{Proposition}

\newtheorem{assumption}[theorem]{Assumption}
\newtheorem{example}{Example}

\begin{document}
\title{Localization and Estimation of Unknown Forced Inputs: A Group LASSO Approach}
	\author{Rajasekhar~Anguluri,~\IEEEmembership{Member,~IEEE,}
		Lalitha~Sankar,~\IEEEmembership{Senior~Member,~IEEE,}
		and~Oliver~Kosut,~\IEEEmembership{Member,~IEEE}
		\thanks{This material is based upon work supported by the National Science Foundation  under  Grant  No.s  OAC-1934766. All the authors are with the School of Electrical, Computer, and Energy
			Engineering, Arizona State University, Tempe, AZ 85281 USA (e-mail:
			\{rangulur,lalithasankar,okosut\}@asu.edu).}}

\maketitle

\begin{abstract}
We model and study the problem of localizing a set of sparse forcing inputs for linear dynamical systems from noisy measurements when the initial state is unknown. This problem is of particular relevance to detecting forced oscillations in electric power networks. We express measurements as an additive model comprising the initial state and inputs grouped over time, both expanded in terms of the basis functions (i.e., impulse response coefficients). Using this model, with probabilistic guarantees, we recover the locations and simultaneously estimate the initial state and forcing inputs using a variant of the group LASSO (linear absolute shrinkage and selection operator) method. Specifically, we provide a tight upper bound on: (i) the probability that the group LASSO estimator wrongly identifies the source locations, and (ii) the $\ell_2$-norm of the estimation error. Our bounds explicitly depend upon the length of the measurement horizon, the noise statistics, the number of inputs and sensors, and the singular values of impulse response matrices.  Our theoretical analysis is one of the first to provide a complete treatment for the group LASSO estimator for linear dynamical systems under input-to-output delay assumptions. Finally, we validate our results on synthetic models and the IEEE 68-bus, 16-machine system. 
\end{abstract}



\section{Introduction}
\label{sec:introduction}
	Low-frequency oscillations in the electric transmission grid are indicative of the type of disturbance afflicting the system. 
	\emph{Natural oscillations}, with frequencies in between 0.1--0.2 Hz, are triggered by random load fluctuations and sudden network switching. In contrast, \emph{forced oscillations} (FOs), with frequencies in between 0.1--15 Hz, result from external inputs injected by malfunctioned or compromised devices, such as power system stabilizers (PSS), generator controllers and exciters, and cyclic loads etc. \cite{BW-KS:17}. FOs remain undamped for longer periods of time, and if not mitigated, they pose a greater risk to the power systems operation, potentially causing blackouts. 
	
	A popular and inexpensive method to mitigate FOs in power systems is to remove the source triggering these oscillations \cite{TH-NMF-PRK-LX:20, SCC-VP-KT:18, BW-KS:17}. This amounts to accurately locating the FO sources. As installing sensors at each potential source is expensive, recent research suggests using phasor measurement unit (PMU) measurements based source localization algorithms. These algorithms range from physics-based energy approaches to completely data-driven approaches \cite{BW-KS:17}; the latter, albeit their impressive performance on test cases, lack theoretical guarantees. This deficiency makes it harder to quantify the performance and limitations of measurement-based methods on what is and is not possible. 
	
	We address the lack of guarantees of existing data-driven approaches by posing the localization problem as a regularized optimization problem---referred to as the group LASSO estimator. The regularization term imposes sparsity constraints on the number of source locations, which is often the case in many practical systems, including power systems \cite{TH-NMF-PRK-LX:20, SM-BW-QZ-FM-XL-SK-EL:16}. The input to our optimization problem are the noisy measurements and dynamical system matrices. It returns the source locations and estimates of unknown initial state and inputs (oscillatory or not) injected by these sources. Formally, we consider
	\begin{align}\label{eq: intro_opt}
	\hspace{-1.0mm}\underbrace{\begin{bmatrix}
	    \what{\mbf{x}}_0 \\ \what{\bsu} 
	\end{bmatrix}}_{\what{\bsbeta}}\!\in \!\argmin_{\substack{\mbf{x}_0,\{\bsu_j\}_{j=1}^m
		}}\norm{\mbf{y}\!-\!\mbf{O}\mbf{x}_0\!-\!\sum_{j=1}^{m}\mbf{J}_j\bsu_j}^2_2 \!+\!\lambda \sum_{j=1}^m\norm{\bsu_{i}}_2,
	\end{align}
	where $\mbf{u}_j =[u_j[0],\ldots,u_j[N]]^\transpose$ is a vector of inputs injected by the $j^{th}$ source, $j \in \{1,\ldots,m\}$, over a discrete time horizon $\mathcal{D}\triangleq \{0,\ldots,N\}$; $\mbf{y}$ is the noisy batch measurements collected over $\mathcal{D}$ from multiple sensors; $\mbf{O}$ and $\mbf{J}_j$ are the observability and forced impulse response matrices, resp; and $\lambda\geq 0$ is the tuning parameter. Let $\bsbeta^*=(\mbf{x}^*_0,\bsu^*_1,\ldots,\bsu^*_m)$ be the unknown ground truth and
	$S\triangleq \{j: \bsu_j^*\ne 0\}\subset \{1,\ldots,m\}$ be the set of active sources. By sparse number of sources, we mean $|S|=m^*\ll m$. Defining $\widehat{S}\triangleq \{j: \widehat{\bsu}_j\ne 0\}$, where $\widehat{\bsu}$ as the estimate, we show that $\widehat{S}=S^*$ and $\|\widehat{\bsbeta}-{\bsbeta}^*\|_2\leq \epsilon$, for any $\epsilon >0$, hold with high probability.

	
In the context of regression models, including linear, logistic, and functional models, a rich literature exists on quantifying the theoretical performance of the group LASSO estimator and its variants; see \cite{KL-MP-SG-ABT:39, NS-RT:12, ML-SVD-PB:08}. However, these works assume $\mbf{J}_j$ and $\mbf{O}$ to be random or to satisfy rather restrictive assumptions, either of them may not hold for $\mbf{J}_j$ and $\mbf{O}$ obtained from linear dynamical systems. Further, $\mbf{J}_j$ associated with the non-zero input $\mbf{u}^*_j$ could be rank deficient, especially if the underlying linear dynamical system is only $d$-\emph{delay} left invertible\footnote{A dynamical system is said to be $d$-delay left invertible if $u_j[k]$ can be uniquely determined from noise-less measurements $\{y[k], y[k+1], · · · , y[k+d]\}$.} \cite{SK-HP-EZ-DSB:11}; this in turn
eliminates the strict convexity property of the objective in \eqref{eq: intro_opt}. As a result, there may exist multiple optimal solutions ($\what{\bsbeta}$); hence, it is not clear if $\what{S}$ is common for all these solutions. We address all these issues by imposing physically meaningful assumptions on $\mbf{O}$ and $\mbf{J}_i$. 
	
	Going beyond the motivating example of forced oscillations in electric power systems, the problem setup in \eqref{eq: intro_opt} is general and the formal results in this paper can be used to localize and reconstruct sparse inputs for a variety of practical engineering systems modeled as linear dynamical systems.  
	
	

	
%
%
%

		\textit{Paper Contributions}:  
		The problem we introduce in \eqref{eq: intro_opt} is distinct from state of the art regularized based optimization methods in seeking to localize inputs and estimate initial state using sufficiently delayed measurements over a block of time. For this model, our main contributions as follows.
	\begin{enumerate}
		\item Our first result is in deriving sufficient conditions under which the following hold with high probability: (i) the estimation error in the $\ell_2$-sense is bounded, and (ii) the localized sources match the true sources. A key contribution is that despite the rank deficiency of model matrices, we guarantee that the group LASSO can localize the sources correctly. For rank deficient matrices $\mbf{J}_i$, we provide estimation guarantees for the delayed inputs (see Section \ref{sec: delayed estimation}). Our result hinges on introducing and thresholding a \emph{mutual incoherence condition} (MIC) on the augmented $\mbf{O}$ and $\mbf{J}_i$ matrices.
		\item The time-domain MIC condition we introduce requires computing correlations among $\mbf{O}$ and $\mbf{J}_i$. This operation is computationally hard especially for large system dimension and estimation horizon $N$. To tackle this hurdle,  we upper bound the time-domain MIC with a frequency-domain MIC. Interestingly, the latter MIC is a sufficient condition if we were to consider a LASSO estimator in the frequency-domain. We also establish a relationship between the performance of the proposed group LASSO approach and the absence of invariant zeros for the sub-system excited by non-zero inputs, and thresholding the frequency domain MIC. 
\item We validate the group LASSO estimator's performance on synthetic data and the IEEE 68-bus, 16-machine system. We implement our estimator using the Alternating Direction Method Multipliers (ADMM) method \cite{SELP09}. 
	\end{enumerate}  
	
\smallskip 
	\textit{Related Literature}: In the context of power systems, model based-approaches, e.g., energy dissipation methods based on frequency domain data and statistical signal processing methods based on AR and ARMA models, are commonly used to localize unknown forced oscillatory inputs. Reference \cite{SC:PV:KT-19} proposes a Bayesian approach to localize sources based on the generators frequency response functions. In \cite{UA:JWP:JF:DD:DT:MD-17}, the pseudo-inverse of a set of system transfer functions are multiplied by a vector of PMU measurements to yield an FO solution vector. In \cite{NZ-MG-SA:17}, the authors leverage the properties of magnitude and phase responses of transfer functions between different buses to identify possible oscillation sources. Instead, completely data-driven methods include the use of novel machine learning methods on the multivariate PMU time series data in \cite{YM-ZY-NL-DS:21} and the robust PCA method in \cite{TH-NMF-PRK-LX:20}, which exploits the low-rank nature of PMU data and also the sparsity of the locations.

%
%
%
	
More broadly, there is a growing research on source localization and unknown state and sparse input reconstruction in dynamical systems based on sparsity constrained optimization methods. The problem of source location in the context of attacks on inputs and sensors is studied in \cite{HF-PT-SD:14} and \cite{fd-fp-fb:12b}. However, these works work with noise-free measurements and do not focus on unknown input estimation. In \cite{MBW-BMS-TLV:10, GJ-RCM:19}, by assuming the knowledge of inputs and using randomly sampled measurements, the authors obtained sample complexity (bounds on the number of measurements) results for reconstructing the initial state with sparsity constraints. 
Instead, the authors in \cite{SS-NJC-RV:15} and \cite{SMF-FG-SG-AYK-DS:19}, consider sparse input and non-sparse state reconstruction using batch wise noise-less measurements and sequential noisy measurements, respectively. However, these works do not address location recovery guarantees for the unknown sources and the initial state. Finally, we acknowledge works in \cite{fp-fd-fb:10y,ML-DG-ZW-HX:19}, where the authors used banks of input observers based residual generation methods to identify source locations in noise-free systems---albeit with strong assumptions, as they do not consider sparse inputs. 

In contrast to these works, we consider a unified framework, based on a LASSO method, to jointly locate the sources, and estimate the sparse inputs along with the unknown initial state. As highlighted in several other non-sparsity based input identification methods \cite{SY-MZ-EF:16,fp-fd-fb:10y,SS-CNH:11}, our results also highlight the role of invariant zeros for sparse input recovery.

\textit{Mathematical Notation}: We denote the vectors and matrices are by boldface lower case and upper case letters. Denote the $d\times d$ identity matrix by $\mbf{I}_d$. Denote the pseudoinverse of $\mbf{X}$ by $\mbf{X}^\dagger$. The rangespace of $\mbf{X}$ is defined by $\mathcal{R}(\mbf{X})=\{\mbf{X}\mbf{z}: \mbf{z} \in \mathbb{R}^m\}$. Given $S \subset \{1,\ldots,m\}$ and $\mbf{x}\in \mathbb{R}^m$, we write $\mbf{x}_S$ for the sub-vector of $\mbf{x}$ formed from the entries of $\mbf{x}$ indexed by $S$. Similarly, we write $\mbf{M}_S$ for the submatrix of $\mbf{M}$ formed from the columns of $\mbf{M}$ indexed by $S$. For $1\leq p<\infty$ and the vector $\mbf{x}=[x_1,\ldots,x_m]$, denote $\norm{\mbf{x}}_p=(\sum_{i=1}^m|x_i|^p)^{1/p}$. Instead, $\norm{\mbf{u}}_\infty=\max_{l}|u_l|$. The $\ell_{a,b}$-mixed-norm, with $a,b \geq 0$, of $\mbf{z}=[\mbf{z}_1^\transpose, \ldots, \mbf{z}_r^\transpose]^\transpose$ is given by $\norm{\mbf{z}}_{a,b}^b=\sum_{j=1}^r\|\mbf{z}_j\|^b_a$. By convention, $\norm{\mbf{z}}_{a,0}\triangleq \sum_{j=1}^{r}I(\|\mbf{z}_j\|_a\ne 0)$, where $I(\cdot)$ is the indicator function, counts the number of non-zero vectors. For a positive integer $m$, we denote $[m]=\{1,\ldots,m\}$. 
		

\section{Problem Setup and Preliminaries}\label{sec: problem_setup}
For a sampled system, we obtain a linear relation between the batch measurements and the initial state and forced inputs. We then formulate a group LASSO optimization problem for the above model to estimate the initial state and inputs, and to locate the unknown sources. 

\vspace{-3.0mm}
	\subsection{Linear dynamics under sparse forced inputs} 
Consider the following continuous-time linear system subjected to external inputs: 
	\begin{align}\label{eq: CT-system}
		\dot{\mbf{x}}_c(t)&=\mbf{A}_c\mbf{x}_c(t)+\mbf{B}_c\mbf{u}^*_c[t], \quad t \in \mathbb{R}, 
	\end{align}
	where $\bs{\Delta}\mbf{x}_c(t)\in \mathbb{R}^n$ and $\mbf{u}^*_c[t]\in \mathbb{R}^{m}$ is the state and input. We assume the input to be sparse, that is $\|\mbf{u}^*_c(t)\|_0\leq m^*<<m$ for all $t\in \mathbb{R}$. In the context of power systems, the state $\bs{\Delta}\mbf{x}_c(t)$ consists of the dynamical states of generators and their control systems, including rotor angles, speed deviations, field excitation voltage, etc. Instead, $\mbf{u}^*_c(t)=[{u}^*_{c,1}(t),\ldots,{u}^*_{c,m}(t)]^\transpose$ is the vector of inputs triggered by the sources of FOs, among which only $m^*$ locations are active. However, our model in \eqref{eq: CT-system}, except for sparsity constraints, is general and allows for multi-dimensional un-modeled exogenous stochastic or deterministic disturbances,  benign faults, or adversarial attacks. 
	
We consider the discrete-time dynamics of \eqref{eq: CT-system} together with a measurement equation: 
\begin{align}
    \mbf{x}[k+1]&=\mbf{A}  \mbf{x}[k]+\mbf{B}\mbf{u}^*[k]\label{eq: DT-system-state}, \\
    \mbf{y}[k]&=\mbf{C}	\mbf{x}[k]+\mbf{v}[k], 	 \quad k=0,1,\ldots, \label{eq: DT-system-measurements}
\end{align}
where $\mbf{A}=e^{\mbf{A}_c\delta t}$, $\mbf{B}=(\int_{0}^{\delta t}e^{\mbf{A}_c\tau}d\tau){\mbf{B}_c}$, and $\delta t$ is the sampling time period, and $\mbf{u}^*[k]=[u_1[k],\ldots,u_{m}[k]]^\transpose$. Further, $\mbf{y}[k]\!=\![y_1[k],\ldots,y_p[k]]^\transpose \!\in\! \mathbb{R}^p$ is the measurement, $\mbf{v}[k]\overset{\text{iid}}{\sim} \mcal{N}(\mbf{0},\sigma^2 \bf{I})$ is noise, and $\mbf{C}\!\in\! \mathbb{R}^{p\times n}$ is the sensor matrix. In Section IV, we consider dynamics in \eqref{eq: DT-system-state} with process noise, and also relax the diagonal covariance assumption on $\mbf{v}[k]$. 

Let $S=\{j: {u}_j[k]\ne 0 \text{ for at least one } k\geq 0\}\subset [m]$ and $S^c=[m]\setminus S$. We refer $S$ and $S^c$ to as the active and inactive set. Partition $\mbf{B}$ as $\mbf{B}=[\mbf{B}_S\,\mbf{B}_{S^c}]$ and $\mbf{u}[k]=[\mbf{u}^\transpose_S[k]\,\, \mbf{u}_{S^c}^\transpose[k]]^\transpose$, with $\mbf{u}^*_{S^c}[k]=[u^*_{i_1}[k],\ldots,u^*_{i_r}[k]]$ and  $\mbf{B}_{S^c}=[\mbf{b}_{i_1},\ldots,\mbf{b}_{i_r}]$, where $i_r \in S^c$ and $r=|S^c|=m-m^*$. Similarly, define $\mbf{u}^*_S[k]$ and $\mbf{B}_S$. Then, the input term in \eqref{eq: DT-system-state} can be written as 
	\begin{align}\label{eq: input representation}
	\begin{split}
	    	    \mbf{B}\mbf{u}^*[k]=\sum_{j=1}^{m}\mbf{b}_j{u}^*_j[k]&=\sum_{j\in S}\mbf{b}_j{u}^*_j[k]+\sum_{j\in S^c}\mbf{b}_j{u}^*_j[k]\\
	    &=\mbf{B}_S\mbf{u}^*_S[k]+\mbf{B}_{S^c}\mbf{u}^*_{S^c}[k].
	\end{split}
	\end{align}
The above representations will play a key role in formulating our group LASSO problem in Section \ref{subsec: group lasso problem}. 
	
	Using \eqref{eq: DT-system-state}-\eqref{eq: DT-system-measurements}, we express the batch measurements
	$\mbf{y}$ (see below) as a linear model with added noise. Define the vectors 
	\begin{align}\label{eq: time collected vectors}
		\mbf{y}\!=\!\begin{bmatrix}\mbf{y}[0] \\ \vdots \\ \mbf{y}[N]\end{bmatrix}, 
		\mbf{v}\!=\!\begin{bmatrix}\mbf{v}[0] \\ \vdots \\ \mbf{v}[N] \end{bmatrix}, \text{ and } 
		\mbf{u}^*_j\!=\!\begin{bmatrix}u^*_j[0] \\ \vdots \\ u^*_j[N] \end{bmatrix},
	\end{align}
	where $\mbf{y}$, $\mbf{v}\in \mathbb{R}^{p(N+1)}$ and $\mbf{u}^*_j \in \mathbb{R}^{N+1}$, for all $j \in S^c$. Here, $N+1$, with $N>0$ is the length of the estimation horizon. We also define the observability matrix $\mbf{O}\in \real^{p(N+1)\times n}$ and the impulse response matrix $\mbf{J}_j\in \real^{p(N+1)\times N+1}$ as
	\begin{align}\label {eq: system matrices}
		\begin{split}
			\hspace{-2.5mm}	\mbf{O}\!=\!\begin{bmatrix}
				\bfC\\	\bfC\bfA \\ \bfC\bfA^2 \\ \vdots \\ \bfC\bfA^{N}
			\end{bmatrix};
			{{\bfJ}}_j\!=\!
			\begin{bmatrix}
				\mbf{H}^{(j)}_0 & \bf{0} & \bf{0} & \ldots & \bf{0}\\
				\mbf{H}^{(j)}_1 & \mbf{H}^{(j)}_0 & \mbf{0} & \ldots & \bf{0}\\
				\mbf{H}^{(j)}_2 & \mbf{H}^{(j)}_1 & \mbf{H}^{(j)}_0 & \ldots & \bf{0}\\
				\vdots & \vdots & \ddots & \ddots & \vdots\\
				\mbf{H}^{(j)}_{N} & \mbf{H}^{(j)}_{N-1} & \ldots & \mbf{H}^{(j)}_{1} & \mbf{H}^{(j)}_{0}\\
			\end{bmatrix},
		\end{split}
	\end{align} 
	where $j \in S\cup S^c$, and the $l$-th impulse response (Markov) parameter, $\mbf{H}^{(j)}_l\in \mathbb{R}^{p\times 1}$, at the $j$-th location is defined as 
	\begin{align}\label{eq: markov parameters}
		\mbf{H}^{(j)}_l := \left\{\begin{array}{lr}
			\mbf{0} & \text{if } l=0,\\
			{\bfC \bfA^{l-1}\mathbf{b}_j} & \text{if } l\geq 1.
		\end{array}\right.
	\end{align}
	
	Let $\mbf{x}[0]=\mbf{x}^*_0$ be the unknown initial state. From \eqref{eq: DT-system-state}-\eqref{eq: DT-system-measurements} and the fact that $\mbf{B}\mbf{u}^*[k]=\sum_{j=1}^{m}\mbf{b}_j{u}^*_j[k]$, we observe that 
	\begin{align}\label{eq: lasso measurement form}
	\mbf{y}=\mbf{O}\mbf{x}_0^*+\sum_{j=1}^m\mbf{J}_j\mbf{u}_j^*+\mbf{v},
	\end{align}
	where $\mathbf{v}\sim \mathcal{N}(\mbf{0},\sigma^2\mathbf{I}_{p(N+1)})$, $\mbf{u}^*_j$ is in \eqref{eq: time collected vectors} and $\mbf{J}_j$ is in \eqref{eq: system matrices}. 

	\subsection{Initial State and Unknown Input Estimation under Sparsity Constraints: A Group LASSO for Approach}\label{subsec: group lasso problem}
	Based the measurement model in \eqref{eq: lasso measurement form}, we introduce the group LASSO estimator to estimate $(\mbf{x}_0^*,\mbf{u}_1^*,\ldots,\mbf{u}_m^*)$ and also the active set $S$. Let $\mbf{J}=[\mbf{J}_1^\transpose,\ldots,\mbf{J}_m^\transpose]$ and 
	$\mbf{u}=[\mbf{u}_1^\transpose,\ldots,\mbf{u}_m^\transpose]$, where $\mbf{u}_j \in \mathbb{R}^{N+1}$. Recall the definition of $\ell_{p,0}$-norm from the notation section, and consider
 	\begin{align}\label{eq: subset selection problem}
	\hspace{-2.7mm}\begin{bmatrix}\what{\bsx}_0 \\ \what{\bsu}\end{bmatrix} = \argmin_{\substack{\bsx_0, \mbf{u}}}\left\lbrace\frac{1}{2T}\norm{\mbf{y}-\mbf{O}\bsx_0-\mbf{J}\bsu}_2^2\!+\!\lambda_{T} \|\mbf{u}\|_{p,0}\right\rbrace, 
	\end{align}
	where the regularization parameter $\lambda_T\!\geq\!0$ and $T=p(N+1)$ is the dimension of $\mbf{y}$ in \eqref{eq: lasso measurement form}. The above problem is called subset (or block-column) selection problem because the optimization problem amounts to finding $\mbf{J}_j$ that contributes to $\mbf{y}$ in \eqref{eq: lasso measurement form}. 


 


	

	
	Unfortunately, \eqref{eq: subset selection problem} is a combinatorial optimization problem and its computationally complexity is exponential in $m$. We circumvent this difficulty by replacing the $\|\mbf{u}\|_{p,0}$ with the $\|\mbf{u}\|_{p,1}$-norm. This is a common relaxation technique widely used in the literature of compressed sensing and statistics; see \cite{YCE:PK:HB-10, MJW:09}. Thus, we end up with the group LASSO problem:
	\begin{align}\label{eq: group LASSO opt}
		\hspace{-2.7mm}	\begin{bmatrix}\what{\bsx}_0 \\ \what{\bsu}\end{bmatrix} \!\in \!\argmin_{\substack{\bsx_0, \bsu }}\left\lbrace\frac{1}{2T}\norm{\mbf{y}-\mbf{O}\bsx_0-\mbf{J}\bsu}_2^2\!+\!\lambda_{T} \|\mbf{u}\|_{p,1}\right\rbrace.
	\end{align}
	
	
	For definiteness, we set $p=2$, although our analysis extends to the case $p\ne 2$. In the literature, $\|\mbf{u}\|_{2,1}=\sum_{j=1}^m\norm{\bsu_{j}}_2$ is referred to as the block or group norm. Our optimization problem in \eqref{eq: group LASSO opt} differs from the traditional group LASSO \cite{NS-RT:12} because we do not penalize $\mbf{x}_0$. This is subtle yet important distinction because in many applications, including power systems, initial state is rarely sparse.  In Section VI, we provide details on how to numerically solve \eqref{eq: group LASSO opt}. Instead, in Section III, for a specific range of $\lambda_T$, we show that the group-norm based regularizer promotes group sparsity in $\widehat{\mbf{u}}$ and that $\what{S}=S$ holds with high probability, where $\what{S}\triangleq\{{j}: \widehat{\mbf{u}}_j\ne 0\}$.

	Due to the presence of additive noise in the measurement vector $\mbf{y}$ in \eqref{eq: lasso measurement form}, neither the estimate 
	$\what{\bsbeta}=(\what{\bsx}_0, \what{\bsu})$ in \eqref{eq: group LASSO opt} need to identically match $\bsbeta^*=(\bsx_0^*, \bsu^*)$ nor does $\what{S}=S$. Thus, we evaluate the quality of our estimates (i.e., the hatted quantities) in a probabilistic sense using the error metrics: 
	\begin{itemize}
		\item $\what{\bsbeta}$ is said to be $\ell_2$-consistent if  $\|\what{\bsbeta}-\bsbeta^*\|_2 \leq o(T)$ with probability at least $1-c_1\exp{(-c_2T)}$, for some $c_1,c_2\!>\!0$. 
		\item $\what{\bsu}$ is said to be location recovery consistent if $\what{S}\!=\! S$ with probability at least $1-c_3\exp(-c_4T)$, for $c_3,c_4>0$. 
	\end{itemize}
	Here $o(T)$ implies that the upper bound on the error tends to zero as $T\to \infty$. 
	The $\ell_2$-error bound ensures that the estimate $\what{\bsbeta}\approx \bsbeta^*$ by increasing $T=p(N+1)$. Instead, the location selection consistency ensures that as as long as $T$ is sufficiently large, $\what{S}$ correctly identifies the true sources of FOs. 
	
\section{Delayed Estimation and Invariant Zeros}\label{sec: delayed estimation}
In this section we cull recent results on the initial state and delayed input recovery using finite number of measurements \cite{AA-DSB:19}, by assuming the knowledge set $S$. These results provide a starting point to prove our main results in Section \ref{sec: main results}. 

We begin by expressing $\mbf{y}$ in \eqref{eq: lasso measurement form} in a slightly different way. From \eqref{eq: input representation}, we have $\mbf{B}\mbf{u}^*[k]=\mbf{B}_S\mbf{u}^*_S[k]+\sum_{j\in S^c}\mbf{b}_j{u}^*_j[k]$. Substituting this fact in \eqref{eq: DT-system-state} and recursively expanding $\mbf{y}[k]$ in \eqref{eq: DT-system-measurements} yields us the following model for $\mbf{y}$ defined in \eqref{eq: time collected vectors}.
	\begin{align}\label{eq: measurement model1*}
	\begin{split}
	    	\mbf{y}&=\underbrace{\begin{bmatrix}
	                \mbf{O} & \mbf{J}_S
	            \end{bmatrix}}_{\triangleq \bs{\Psi}_S}\underbrace{\begin{bmatrix}
	                \mbf{x}^*_0\\
	                \mbf{u}^*_S
	            \end{bmatrix}}_{\triangleq \bsbeta^*_S}+\sum_{j\in S^c}{\mbf{J}}_j\bsu^*_j+\mbf{v},
	\end{split}
	\end{align} 
where $\mbf{u}^*_S$ and
$\mbf{J}_S\in \real^{p(N+1)\times m^*(N+1)}$ are defined as 
	\begin{align}\label{eq: system matrices1}
		\begin{split}
        \hspace{-3.0mm}\mbf{u}^*_S\!=\!\begin{bmatrix}
		\mbf{u}^*_S[0]\\\mbf{u}^*_S[1] \\ \mbf{u}^*_S[2] \\ \vdots \\ \mbf{u}^*_S[N]
			\end{bmatrix};
			{{\bfJ}}_S\!=\!
			\begingroup 
\setlength\arraycolsep{1.5pt}\begin{bmatrix}
				\mbf{H}^{(S)}_0 & \bf{0} & \bf{0} & \ldots & \bf{0}\\
				\mbf{H}^{(S)}_1 & \mbf{H}^{(S)}_0 & \mbf{0} & \ldots & \bf{0}\\
				\mbf{H}^{(S)}_2 & \mbf{H}^{(S)}_1 & \mbf{H}^{(S)}_0 & \ldots & \bf{0}\\
				\vdots & \vdots & \ddots & \ddots & \vdots\\
				\mbf{H}^{(S)}_{N} & \mbf{H}^{(S)}_{N-1} & \ldots & \mbf{H}^{(S)}_{1} & \mbf{H}^{(S)}_{0}\\
			\end{bmatrix}\endgroup ,
		\end{split}
	\end{align} 
	with $\mbf{H}_0^{(S)}=\mbf{0}_{p\times m^*}$ and $\mbf{H}^{(S)}_l={\bfC \bfA^{l-1}\mathbf{B}_S}$, for all $l\geq 1$. Note that $\mbf{y}$ in \eqref{eq: lasso measurement form} and \eqref{eq: measurement model1*} are exactly the same. Importantly, $\mbf{u}_S^*$ in \eqref{eq: system matrices1} is a concatenation of inputs $\mbf{u}^*_S[k]$ associated with $S$ from $k=0$ (top) to $N$ (bottom), but not a concatenation of $\mbf{u}_j^*$ in \eqref{eq: time collected vectors}, for all $j \in S$. 
	
To show that the group LASSO is location recovery consistent, or $\what{S}=S$ holds with high probability,  $\bs{\Psi}_S=\begin{bmatrix}
	\mbf{O} & \mathbf{J}_S
\end{bmatrix}$ in \eqref{eq: measurement model1*} should be of full column rank. To see this, suppose that $\sigma^2\approx 0$ and that we know $S$. Then, by substituting $\mbf{u}_j^*=\mbf{0}$, for all $j \in S^c$, and $\mbf{v}=\mbf{0}$ in  
$\mbf{y}$ in \eqref{eq: measurement model1*}, it follows that 
	\begin{align}\label{eq: noiseless partitioned measurements}
	\begin{split}
		\mbf{y}&= \bs{\Psi}_S\bsbeta^*_S. 
	\end{split}
\end{align}
Thus for a rank deficient $\bs{\Psi}_S$, we cannot perfectly recover $\bsbeta^*_S=(\mbf{x}_0^*, \mbf{u}_S^*[0],\ldots,\mbf{u}^*_S[N])$ even with noise-free measurements and with the knowledge of $S$. However, unfortunately, unlike the model matrices, such as random design and Fourier basis matrices, considered in signal processing and statistics applications, $\bs{\Psi}_S$ could be rank deficient. This is so because
	system in \eqref{eq: DT-system-state}-\eqref{eq: DT-system-measurements} may not be initial state and input observable \cite{SK-HP-EZ-DSB:11}; that is, either $\mbf{O}$ or $\mbf{J}_S$ is rank deficient, or both $\mbf{O}$ and $\mbf{J}_S$ have full ranks, but $[
	\mbf{O} \,\, {\mathbf{J}}_S]$ is rank deficient. 
	
From the foregoing discussion, it is clear that recovering $\bsbeta_S^*$ and full rank of $\bsPsi_S$ are intimately connected. Interestingly, for $d$-delay invertible linear systems, 
even when $\bsbeta_S^*$ is not recoverable, a portion of it is perfectly recoverable \cite{SK-HP-EZ-DSB:11, AA-DSB:19}. In fact, we can recover $\bsbeta_{S,[0:N-d]}^*=(\mbf{x}^*_0,\mbf{u}_S^*[0],\ldots,\mbf{u}_{S,[N-d]})$, where $N\geq d$, from $\mbf{y}^\transpose=[\mbf{y}^\transpose[0],\ldots, \mbf{y}^\transpose[N]]$
Here, $d\geq 0$ is called \emph{delay} and we refer $\bsbeta^*_{S,[0:N-d]}$ to as the delayed input. As a result, we show that a specific sub-matrix of  $\bsPsi_{S}$ has full column rank even when $\bsPsi_{S}$ is rank deficient.

We formalize the notion of $d$-delay. Let $\mbf{x}_0\!=\!\mbf{0}$ to note that $\bsPsi_S\!=\!\mbf{J}_S$ and $\bsbeta^*_S\!=\!\mbf{u}^*_S$. Substituting $\mbf{J}_S$ \eqref{eq: system matrices1} in \eqref{eq: noiseless partitioned measurements}, yields 
\begin{align}
	\underbrace{\begin{bmatrix}
		\mbf{y}[0]\\
		\mbf{y}[1]\\
		\vdots \\
		\mbf{y}[N]
	\end{bmatrix}}_{\mbf{y}_N}&\!=\!\underbrace{\left[\begin{array}{c|ccc}
\mbf{H}^{(S)}_0 & \mbf{0} & \ldots & \mbf{0}\\
\hline 
\mbf{H}^{(S)}_1  & \mbf{H}^{(S)}_0 & \ldots & \mbf{0}\\
\vdots & \ddots & \ddots & \vdots \\
\mbf{H}^{(S)}_N & \mbf{H}^{(S)}_{N-1} & \ldots & \mbf{H}^{(S)}_0 
\end{array}
\right]}_{\triangleq \bfJ_{S,[N:0]}}
\underbrace{\begin{bmatrix}
	{\mbf{u}}^*_S[0]\\
	\hline 
	{\mbf{u}}^*_S[1]\\
	\vdots \\
	{\mbf{u}}^*_S[N]
\end{bmatrix}}_{{\mbf{u}}^*_{S,[0:N]}}.\label{eq: expanded l measurements1}
\end{align}
Notice that $\mbf{J}_S=\mbf{J}_{S,[N:0]}$ and $\mbf{u}^*_S=\mbf{u}^*_{S,[0:N]}$. Define 
\begin{align}
   \mbf{J}_{S,[N:0]} =\left[\begin{array}{c|c|cc}
	\mbf{M}^{(S)}_N & \mbf{M}^{(S)}_{N-1} & \ldots & \mbf{M}^{(S)}_{0}
\end{array}
\right],\label{eq: expanded l measurements2}
\end{align}
where $\mbf{M}^{(S)}_l$ denotes the $l^{th}$ block column of $\mbf{J}_{S,[N:0]}$ labeled right $(l=0)$ to left $(l=N)$. By construction $\bfJ_{S,[N:0]}$ is rank deficient because $\mbf{H}^{(S)}_0=\mbf{0}$ in \eqref{eq: expanded l measurements1}. Thus, we cannot recover ${\mbf{u}}^*_S[N]$ using $\mbf{y}_N$. Further, in several practical applications, $\mbf{H}^{(S)}_1=\bfC\bfB_S=\bfzeros$ (or has non-full column rank). This is because sensors may not be located at the inputs. For e.g., in power systems, bus level PMUs do not directly measure PSS's output. Thus it is impossible to recover ${\mbf{u}}^*_S[N-1]$ using $\mbf{y}_N$. 


\begin{definition}{\bf \emph{(System delay)}}\label{def: system delay} For a non-negative integer $d\geq 0$, let $\mbf{{J}}_{S,[d:0]}$ be defined as in \eqref{eq: expanded l measurements1}. System in \eqref{eq: DT-system-state}-\eqref{eq: DT-system-measurements}, with $\mbf{x}_0^*=\mbf{0}$, $\mbf{u}_{S^c}^*=\mbf{0}$ and $\sigma^2=0$, is \emph{$d$-delay} left invertible if
\begin{align}\label{eq: minimum input delay}
\Rank (\mbf{{J}}_{S,[d:0]}) - \Rank (\mbf{{J}}_{S,[d-1:0]})=m^*, 
\end{align}
for $\mbf{{J}}_{S,[d:0]}$ defined in \eqref{eq: expanded l measurements1} and $m^*$ is the dimension of ${\mbf{u}}^*_S[k]$. The smallest $d$ that satisfies \eqref{eq: minimum input delay} is denoted as $\eta_S$. \QEDB
\end{definition}

\smallskip 
Throughout we assume $d=\eta_S$ and set $d\triangleq \infty$ if \eqref{eq: minimum input delay} does not hold for any $d\geq 0$. Suppose that $\eta_S<\infty$. Then, from the rank properties of partitioned matrices \cite{AA-DSB:19}, it follows that $\mbf{M}^{(S)}_d$ in \eqref{eq: expanded l measurements1} is full column rank. Thus, there exists a matrix $\mbf{S}$ such that $\mbf{S}\mbf{y}_d={\mbf{u}}^*_S[0]$. We may recover ${\mbf{u}}^*_S[1]$ using the residual $\what{\mbf{y}}_{d+1}\triangleq \mbf{y}_{d+1}-\mbf{M}^{(S)}_{d+1}{\mbf{u}}^*_S[0]$. In fact,  $\mbf{S}\what{\mbf{y}}_{d+1}={\mbf{u}}^*_S[1]$. By iterating this procedure, we can recover inputs in  ${\mbf{u}}^*_{S,[0:N-d]}\triangleq[(\mbf{u}^*_{S}[0])^\transpose,  \ldots, (\mbf{u}^*_{S}[N-d])^\transpose]^\transpose$ using $\mbf{y}_{N}$. 

\smallskip 
We relax $\mbf{x}_0^*=\bfzeros$ assumption and extend the rank condition in \eqref{eq: minimum input delay} to recover jointly $\bsbeta_{S,[0:N-d]}^*=(\mbf{x}^*_0,{\mbf{u}}^*_{S,[0:N-d]})$, as a whole rather than sequentially, using $\mbf{y}_N$. First, we define the smallest delay for recovering $\mbf{x}^*_0$ in the presence of input: 
	\begin{align}\label{eq: smallest index}
	\mu_S&\!\triangleq\! \min\{d \geq 0: \Rank ([
	\mbf{O}_d \,\, {\mbf{J}}_{S,[d:0]}])\!-\!\Rank (\bfJ_{S,[d:0]})\!=\!n\}, 
\end{align}
where $\mbf{O}_d=[\mbf{C}^\transpose \,\, (\mbf{C}\mbf{A})^\transpose \, \ldots, \, (\mbf{C}\mbf{A}^{d})^\transpose]$, and $n$ is the dimension of $\mbf{A}$. The rank condition in \eqref{eq: smallest index} says that $\mbf{O}_d$ has full column rank $(=n)$ and that the columns in $\mbf{O}_d$ are linearly independent of columns in ${\mbf{J}}_{S,[d:0]}$. This condition is stronger than system in \eqref{eq: DT-system-state}-\eqref{eq: DT-system-measurements} being observable, as shown below: 

\begin{example}
	Let $\mbf{A}=\begin{bmatrix}
		1 & 2;0 & 3
	\end{bmatrix}$, $\mbf{B}_S=\begin{bmatrix}
	2 & 3
\end{bmatrix}^\transpose$, and $\mbf{C}=\begin{bmatrix}
1 & 0
\end{bmatrix}$. Then $\eta_S=1$ and $\Rank \mbf{O}_l=2$, for any $\ell \geq 2$; that is, the system is observable. However, $\mu_S = \infty$ because the rank condition in \eqref{eq: smallest index} does not hold. This is to be expected because the second column of $\mbf{A}$ is identical to $\mbf{B}_S$. \QEDB
\end{example}

\smallskip 
Let $\mbf{M}_l^{(S)}$ be as in \eqref{eq: expanded l measurements2}. For $N\geq d \geq 0$, consider
\begin{align}\label{eq: Psi}
	\underbrace{\begin{bmatrix}
			\mbf{O} & \bsJ_{S}
	\end{bmatrix}}_{\bsPsi_S}&=\begin{bmatrix}
		\smash[b]{\underbrace{\begin{matrix}\mbf{O}& \mbf{M}^{(S)}_N &  \ldots & \mbf{M}^{(S)}_{d}\end{matrix}}_{\bsPsi_{S,[N:d]}}} & \smash[b]{\underbrace{\begin{matrix}\mbf{M}^{(S)}_{d-1} & \ldots & \mbf{M}^{(S)}_{0}\end{matrix}}_{\bsPsi_{S,[d-1:0]}}}
	\end{bmatrix}. 
\end{align}
Let $\bsPsi^+_S$ be the pseudo inverse of $\bsPsi_S$. The proposition below states conditions under which we can recover $(\mbf{x}_0^*, {\mbf{u}}^*_{S,[0:N-d]})$. 

\smallskip 
\begin{proposition}\label{prop: sub-matrix full rank}
Suppose that $\eta_S$ in \eqref{eq: minimum input delay} and $\mu_S$ in \eqref{eq: smallest index} are finite. Then, for $N\geq\max\{\eta_S,\mu_S\}$ with $d\geq \eta_S$, we have
	\begin{enumerate}
		\item $\bsPsi_{S,[N:d]}$ defined in \eqref{eq: Psi} has full column rank. 
		\item $\mathcal{R}(\bsPsi_{S,[N:d]})\cap \mathcal{R}(\bsPsi_{S,[d-1:0]})=\{0\}$. 
	\end{enumerate}
Moreover, for $t_S\triangleq (N-d+1)m^*$ and $m^*=|S|$, we have
\begin{align}\label{eq: Pi selection above}
	\begin{bmatrix}
		\bfxs_0\\
		\bfust_{S,[0:N-d]}
	\end{bmatrix}\!=\!&\underbrace{\begin{bmatrix}
			\mbf{I}_{n+t_S} & \mbf{0}_{(n+t_S)\times dm^*}
	\end{bmatrix}}_{\wtilde{\boldsymbol{\Pi}}_{S,[0:N-d]}}\bsPsi_S^+\mbf{y}. 
\end{align}
\end{proposition}
The proof of this fact is given in \cite[Theorem 7]{AA-DSB:19}. Part (1) of proposition states that the sub-matrix $\bsPsi_{S,[N:d]}$ has full rank even when $\bsPsi_{S,[N:0]}$ is rank deficient. This fact plays a vital role in the performance analysis of the group LASSO estimate. 

For Proposition \ref{prop: sub-matrix full rank} to hold, we require $\eta_S,\mu_S<\infty$. Using the notion of zeros and rank of the system matrix (see below), we state verifiable conditions to check if $\eta_S,\mu_S<\infty$. For all $z \in \mathbb{C}$, define the system and transfer matrix: 
\begin{align}
	\mathcal{Z}_S[z]&\triangleq \begin{bmatrix}
		z\mbf{I}-\bfA & -\mbf{B}_S\\
		\bfC & \mbf{0}
	\end{bmatrix} \text{ and } \label{eq: rosenbrock matrix}\\
	\mathcal{G}_{S}[z]&\triangleq\bfC(z\mbf{I}-\bfA)^{-1}\bfB_S, \quad z \notin \text{spec}(\mbf{A})\label{eq: TF matrix}, 
\end{align}
where $\text{spec}(\bfA)$ is the multiset of eigenvalues of $\mbf{A}$. Define the normal ranks of $\mathcal{Z}_S[z]$ and $\mathcal{G}_S[z]$, respectively, as $\text{nRank} \mathcal{Z}_S\triangleq \max_{z \in \mathbb{C}} \Rank \mathcal{Z}_S[z]$ and $\text{nRank} \mathcal{G}_S\triangleq \max_{z \in \mathbb{C}} \Rank \mathcal{G}_S[z]$.  A number $z_0\in \mathbb{C}$ is called the invariant zero of $\sysmatS$ if $\Rank \mathcal{Z}_S[z_0]< \text{nRank} \mathcal{Z}_S$. If $\sysmatS$ has invariant zeros, there exists $\mbf{u}^*_S\ne 0$ and $\mbf{x}_0\ne 0$ such that (noise-free) $y[k]=\mbf{0}$, for all $k\geq 0$ \cite{BDOA:MD-09}. (Thus, we cannot distinguish between non-zero and zero inputs from $\mbf{y}_N$ alone.) Hence, $\eta_S,\mu_S=\infty$.  

\begin{lemma}\label{lma: invariant zeros and conditions one and two}
Let $\sysmatS$ has no invariant zeros. Then, (i) $\mu_S<\infty$ and, for $N\geq \mu_S$, system in \eqref{eq: DT-system-state}-\eqref{eq: DT-system-measurements} is initial state observable; and (ii) if $\mathrm{nRank} \mathcal{Z}_S=n+m^*$, $\eta_S<\infty$.  
\end{lemma}
\smallskip 
A proof for the statement (i) can be found in \cite[Proposition 5]{AA-DSB:19}. Instead, the statement (ii) follows from \cite[Theorem 1, pp. 227]{SK-HP-EZ-DSB:11}. Thus, if $\sysmatS$ satisfies conditions in Lemma \ref{lma: invariant zeros and conditions one and two}, the assumptions in Proposition \ref{prop: sub-matrix full rank} hold. Hence, the sub-matrix $\bsPsi_{S,[N:d]}$ has full rank and we can recover $(\mbf{x}_0^*,\mbf{u}^*_{S:[0:N-d]})$. 

\section{Location Recovery and Estimation Consistency of the group LASSO Estimator}\label{sec: main results}
We theoretically investigate the performance of the group LASSO estimator in \eqref{eq: group LASSO opt} using the previously stated results for the delayed input estimation. Our results generalize the existing group LASSO's guarantees for static (or non-dynamical) systems \cite{KL-MP-SG-ABT:39, HL-JZ:09} to the dynamical systems with delay $d\geq 0$.

Recall that the estimate in \eqref{eq: group LASSO opt} is $(\what{\mbf{x}}_0,\what{\mbf{u}}_1,\ldots,\what{\mbf{u}}_m)$, where $\what{\bfu}_j=[\hat{u}_j[0],\ldots,\hat{u}_j[N]]^\transpose$. For any $R\subset \{1,\ldots,m\}$, we define $\what{\bfu}_R[k]=[\what{u}_{r_1}[k],\ldots,\what{u}_{r_{|R|}}[k]]$, for all $k \geq 0$ and $r_j \in R$. In words, we group the estimated inputs associated with the set $R$. Further, define $\what{\bfu}_R^\transpose=[\what{\bfu}^\transpose_R[0],\ldots,\what{\bfu}^\transpose_R[N]]$. Thus, we can compare $\what{\bsbeta}_S=(\what{\bfx}_0,\what{\bfu}_S)$ and $\bsbeta^*_S=(\mbf{x}^*_0,{\bfu}^*_S)$ in \eqref{eq: measurement model1*}.

 Recall that $S=\{j: {\mbf{u}}^*_j \ne 0 \}$ and $\what{S}=\{j: \what{\mbf{u}}_j \ne 0 \}$. 
 We derive conditions under which (i) $\what{S}=S$ and (ii) $\|\bsbeta^*_{S,[0:N-d]}-\what{\bsbeta}_{S,[0:N-d]}\|_2\leq \epsilon$, for any $\epsilon > 0$, hold with high probability. To this aim, we make a few assumptions. If $d=0$ and $\mbf{x}_0^*=\mbf{0}$, these assumptions reduce to the standard group LASSO's assumptions \cite{MJW:09}. 
 
 	\begin{assumption} {\bf \emph{(Identifiability and mutual incoherence conditions)}} 
		Consider the following conditions: 
		%
		\begin{itemize}
\item [(A1)]\textit{Group normalization}: The column block matrices $\mbf{O}$ and $\mbf{J}_i$ in \eqref{eq: system matrices} satisfy the group normalization condition: 
			\begin{align}\label{eq: group normalization}
		\max\left\lbrace{\norm{\bsO}_2},{\norm{\bsJ_1}_2},\ldots,{\norm{\bsJ_m}_2}\right\rbrace\leq C\sqrt{T}<\infty . 
			\end{align}
			\item [(A2)]\textit{Least singular value}:
			Let $\bsPsi_{S,[N:d]}$ and $\bsPsi_{S,[d-1:0]}$ be as in \eqref{eq: Psi}, and define $\mbf{M}=[\mbf{I}-\bsPsi_{S,[d-1:0]}\bsPsi_{S,[d-1:0]}^+]$. Then
			\begin{align}\label{eq: strong-invertibility}
			 \norm{\left(\frac{\bsPsi_{S,[N:d]}^\transpose \mbf{M}\bsPsi_{S,[N:d]}}{T}\right)^\dagger}_2\leq \frac{1}{c_\text{min}} <\infty. 
			\end{align}
			\item [(A3)] \textit{Mutual incoherence}: There exists some $\alpha \in [0,1)$, referred to as \textit{"mutual incoherence"} parameter, such that  
			\begin{align}\label{eq: mutual incoherence}
				\mathrm{MIC}\triangleq \max_{j \in S^c}\norm{\mbf{J}_j^\transpose\bsPsi_S(\bsPsi_S^\transpose\bsPsi_S)^+}_2\leq \alpha/m^*.
			\end{align}
		\end{itemize}
	\end{assumption}
	
Assumption (A1) holds for asymptotically stable systems; that is, the eigenvalues of $\mbf{A}$ completely lie inside a complex unit circle. Assumption (A2) enforces conditions on the model identifiability; that is, the uniqueness of the delayed estimate $\what{\mbf{u}}_{S,[0:N-d]}$ but not necessarily on  $\what{\mbf{u}}_{S,[0:N]}$ as we shall see later. Assumption (A2) is satisfied for $d$-delay systems because $\bsPsi_{S,[N:d]}$ has full column rank if $N\geq d$ (see Section III). 

Assumption (A3) is satisfied if $\bsPsi_S$ and $\mbf{J}_j$ are orthogonal ($\mbf{J}_j^\transpose \bsPsi_S=\bfzeros$, for all $j \in S^c$). Orthogonality is restrictive as number of inputs can be more than outputs, or any column of $\bfB_S$ in \eqref{eq: DT-system-state} can be a linear combination of $\mbf{b}_j$, for $j \in S^c$. Nonetheless, (A6) imposes a type of "approximate" orthogonality between $\mbf{J}_j$, where $j \in S^c$, and $\bsPsi_S$. We quantify this approximation using the parameter $\alpha$. The $\ell_2$-norm bound in \eqref{eq: mutual incoherence} could be conservative as the bound depends on $m^*$. This dependence can be avoided by working with the $\ell_1$-norm bound; that is, $\max_{j \in S^c}\norm{\mbf{J}_j^\transpose\bsPsi_S(\bsPsi_S^\transpose\bsPsi_S)^+}_1\leq \alpha$. However, we stick with \eqref{eq: mutual incoherence} as it is useful to derive an upper bound on MIC in \eqref{eq: mutual incoherence} using the system transfer function. In simulations, we study the conservatism incurred due to $\ell_2$-norm based MIC. 

	\begin{theorem}{\bf \emph{(Location recovery consistency)}}\label{thm: location-consistency} Consider the model \eqref{eq: measurement model1*} satisfying assumptions (A1)-(A3) with the active set $S= \{1,\ldots, m^*\}$. For some $\delta>0$ suppose that we select
		\begin{align}\label{eq: optimal lambda}
			\lambda_T=\frac{\sqrt{32}C\sigma}{1-\alpha}\left\lbrace\sqrt{\frac{(N+1)c_1+\log(m-m_0)}{T}}+\frac{\delta}{2} \right\rbrace,  
		\end{align}
		where $c_1=\log(5)$. Then, the following hold with probability at least $1-4\exp(-T\delta^2/2)$. 
		
	\begin{enumerate}[label=(\alph*)]
	\item \textit{(Non-uniqueness)}: For $d>0$, there are infinitely many solutions of \eqref{eq: group LASSO opt}; however, if $d=0$, $\what{\bsbeta}$ in \eqref{eq: group LASSO opt} is unique. 
	\item \textit{(No false inclusion)}: For all $d\geq 0$, the support set of any optimal estimate $\what{\bsbeta}$ is contained with in the true support set; that is, $\what{S}\subset S$.
	\item ($\ell_{\infty}$ bounds): The delayed inputs satisfy the following bound:  $\max_{j \in S}\|\widehat{\mbf{u}}_{j,[0:N-d]}\!-\!\mbf{u}^*_{{j},[0:N-d]}\|_\infty\leq \beta_\text{min}$, where 
		\begin{align}\label{eq: beta min}
			\beta_\text{min}&\!=\! \frac{\sigma}{\sqrt{c_{\min}}}\left\lbrace\sqrt{\frac{2\log((N-d+1)m^*)}{T}}+\delta\right\rbrace\nonumber \\
			&\,\, +\lambda_T\norm{\boldsymbol{\Pi}_{S,[0:N-d]}\left(\frac{\bsPsi^\transpose_S\bsPsi_S}{T}\right)^+}_{\infty}, 
		\end{align}
		$\boldsymbol{\Pi}_{S,[0:N-d]}=[\mbf{0}_{t_S\times n} \,\mbf{I}_{t_S\times dm^*}\,\mbf{0}_{t_{S^c}\times t_{S^c}}]$ and $t_{S}=(N-d+1)m^*$. 
		\item \textit{(Minimum input magnitude and no false exclusion)}: If $\min_{j\in S}\|\mbf{u}_{j,[0:N-d]}^*\|_{\infty}\geq \beta_\text{min}$, we have $\what{S}=S$.
	\end{enumerate}
		\begin{proof}
			See Appendix. 
		\end{proof}
	\end{theorem}
	
		\begin{corollary}\label{cor: l2-consistency} Consider  $\what{\bsbeta}_{S,[0:N-d]}=(\what{\bfx}_0,\what{\bfu}_{S,[0:N-d]})$ and $\bsbeta^*_{S,[0:N-d]}=(\mbf{x}^*_0,{\bfu}^*_{S,[0:N-d]})$. Let $c_1=\log(5)$ and $t_S=(N-d+1)m^*$. Under the assumptions of Theorem \ref{thm: location-consistency}, with probability at least $1-\exp(-\delta^2T/2)$, we have
		\begin{align}\label{eq: l2 bound corollary}
			\norm{\bsbeta^*_{S,[0:N-d]}-\what{\bsbeta}_{S,[0:N-d]}}_2&\!\leq\!\nonumber \\
			&\hspace{-21.0mm} \frac{2\sigma}{\sqrt{c_{min}}}\left\lbrace\sqrt{\frac{2c_1(n+t_S)}{T}}+\delta\right\rbrace+ \lambda_T\sqrt{\frac{m^*}{Tc_{\min{}}}}, 
		\end{align}
	\end{corollary}

We use the \emph{primal-dual witness} technique of Wainwright \cite{MJW:09, MJW:19, HL-JZ:09} to prove Theorem \ref{thm: location-consistency}. 


Part (a) in Theorem \ref{thm: location-consistency} states that the group LASSO estimate $\bsbeta$ is non-unique unless the sub-system realized by $\sysmatS$ has zero delay. This is because, for $N>d>0$, the sub-matrix $\bsPsi_{S,[N:d]}$ in \eqref{eq: Psi} has full rank, but not $\bsPsi_{S}$. However, Part (b) in Theorem \ref{thm: location-consistency} states that $\widehat{S}\subseteq S$, for any optimal estimate $\bsbeta$ in \eqref{eq: group LASSO opt}. Thus, the estimated inputs restricted to the complement set are zero: $\widehat{\bfu}_{j^c}=\bfzeros$, for all $j \in S^c$. Thus, the  non-uniqueness of the optimal solution does not effect the location consistency of the group LASSO estimator. 

Part (d) in Theorem \ref{thm: location-consistency} (d)---a consequence of the $\ell_\infty$ norm bound in part (b)---says that for $\widehat{S}= S$ to hold (i.e., to detect true inputs correctly) , the true non-zero input signal strength should not be too small, precisely, smaller than $\beta_{min}$ in \eqref{eq: beta min}.  The probabilistic result in Theorem \ref{thm: location-consistency} also helps determine the number of measurements ($N$) or sensors ($p$) required to achieve certain amount of performance. Let us simplify $\lambda_T$ in \eqref{eq: optimal lambda}  to comment on its scaling. By substituting $T=p(N+1)$ and assuming that $\log(m-m*)/(N+1)>>c_1$, we have
\begin{align}\label{eq: simplified lambda}
	\lambda_T=O\left(\sqrt{\frac{\log(m-m^*)}{p(N+1)}}+\frac{\delta}{2}\right). 
\end{align}
For $p=1$, $\lambda_T$ in \eqref{eq: simplified lambda} reduces to that of $\lambda_T$ for the traditional LASSO problem \cite{MJW:09}.  Thus, the term $c_1(N+1)/T$  in \eqref{eq: optimal lambda} takes into consideration the number of unknowns in $\bsu^*_j$, and $p$ in $p(N+1)$ accounts for the number of sensors. 

The choice of $\lambda_T$ plays an important role in determining if Theorem \ref{thm: location-consistency} (c) (that is, $\what{S}=S$) holds. In fact, the smaller the $\lambda_T$, the smaller the minimum threshold $\beta_\text{min}$. Interestingly, for $\lambda_T=0$, which happens, say, when $\sigma=0$, the optimization problem in \eqref{eq: group LASSO opt} reduces to the standard ordinary least squares (OLS) problem. Thus, there is no shrinkage of input estimates toward zero. Further, $\lambda_T$ does not depend on $c_{min}$ in \eqref{eq: strong-invertibility} but depends on the group normalization constant $C$ in \eqref{eq: group normalization} and the mutual incoherence parameter $\alpha$ in \eqref{eq: mutual incoherence}. 

To understand the role of $C$ on $\lambda_T$, and ultimately on $\beta_\text{min}$, let $d=0$ and note that $\bsPsi_S$ full rank. Assuming (A1) holds with equality, from the standard norm inequalities, we have 
\begin{align*}
    \kappa_1+\lambda_T\sqrt{\kappa_2}/C^2\geq\beta_\text{min}\geq \kappa_1+\lambda_T/(C^2\sqrt{\kappa_2}),
\end{align*}
where $\kappa_1$ is the first term on the right side of the equality in \eqref{eq: beta min} and $\kappa_2=(N+1)m^*$ is the dimension of $\mbf{u}^*_S$. Noting that $\lambda_T$ is proportional to $C$, we see that $\beta_\text{min}=\kappa_1+O(\sqrt{\kappa_2}/C)$. As expected, larger values of $C$ results in smaller $\beta_\text{min}$ because the effective signal strength of $\bsPsi_S\mbf{u}^*_S$ is large. Instead, smaller values of $C$ results in higher $\beta_\text{min}$, thereby requiring $\mbf{u}_S^*$ to be large. If not, the strength of $\bsPsi_S\mbf{u}^*_S$ decreases. Finally, from \eqref{eq: optimal lambda}, we observe that $\lambda_T$ is an increasing function of $\alpha\in [0,1)$; thus, higher the $\alpha$ larger is the $\beta_\text{min}$. Recall that $\alpha$ is large if $\mbf{J}_j$, for $j \in S^c$, is highly correlated with $\bsPsi_S$.


We now comment on the $\ell_2$-error bound between $\bsbeta^*_{S,[0:N-d]}$ and $\what{\bsbeta}_{S,[0:N-d]}$ given in Corollary \ref{thm: location-consistency}. First, the error bound depends on the number of unknown parameters $n+t_S=n+(N-d+1)m^*$, i.e., the dimension of the initial state and delayed input. Letting $T=p(N+1)\gg n$, we observe that the first term of the bound in \eqref{eq: l2 bound corollary} scales as $O(\tilde{c}(\sqrt{m^*/p}+\delta))$, where $\tilde{c}=2\sigma/\sqrt{c_\text{min}}$. Thus, more PMUs result in less error. However, the bound is loose for large values of $\lambda_T$. 
To remedy this shortcoming, we consider the following OLS estimate: 
\begin{align}\label{eq: OLS estimate}
\what{\bsbeta}^{(OLS)}_{\widehat{S},[0:N-d]}\triangleq\wtilde{\boldsymbol{\Pi}}_{\widehat{S},[0:N-d]}(\bsPsi_{\widehat{S}}^+\mbf{y}),
\end{align}
where $\wtilde{\boldsymbol{\Pi}}_{\widehat{S},0:N-d}$ is defined similar to $\wtilde{\boldsymbol{\Pi}}_{{S},0:N-d}$ in \eqref{eq: Pi selection above}. We present the second main result of this section: an oracle bound on the error $\|\bsbeta^*_{S,[0:N-d]}-\what{\beta}^{(OLS)}_{\widehat{S},[0:N-d]}\|_2$.

\begin{theorem}{\bf \emph{($\ell_2$-consistency: oracle bounds)}}\label{thm: l2-consistency} Suppose that the hypotheses in Theorem \ref{thm: location-consistency} hold. Then, for any $\delta,\delta_1>0$, with probability at least $1-4\exp(-T\delta^2/2)-\delta_1$,
		\begin{align}\label{eq: l2 bound oracle}
			\norm{\bsbeta^*_{S,[0:N-d]}-\what{\bsbeta}^{(OLS)}_{\what{S},[0:N-d]}}_2&\!\leq\!\frac{4\sigma}{\sqrt{c_{min}}}\left\lbrace\sqrt{\frac{(n+t_S)}{T}}\right\rbrace\nonumber \\
			& \hspace{-8.0mm} +\frac{2\sigma}{\sqrt{c_{min}}}\left\lbrace\sqrt{\frac{1}{T}\log\left(\frac{1}{{\delta_1}}\right)}\right\rbrace,
		\end{align}
\end{theorem}
The proof is in Appendix. Similar to the bound in Corollary \ref{cor: l2-consistency}, the first term in \eqref{eq: l2 bound oracle} is $O(\tilde{c}(\sqrt{m^*/p}))$; however, the second term in \eqref{eq: l2 bound oracle} does not depend on $\lambda_T$ and it approaches zero as $T\to \infty$. Thus, the overall error is dictated by $m*/p$. We call the bound in \eqref{eq: l2 bound oracle} as the oracle because the bound holds for $\what{\bsbeta}^{(OLS)}_{{S},[0:N-d]}$, albeit with probability $1-\delta_1$.

\vspace{-1.0mm}
\subsection{Extensions of group LASSO guarantees to noisy dynamics}\label{sec: noisy dynamics}
We extend our results in Theorems \ref{thm: location-consistency} and \ref{thm: l2-consistency} to the setting where system in \eqref{eq: DT-system-state}-\eqref{eq: DT-system-measurements} is affected by both state and measurement noises. We also relax the diagonal covariance structure of the measurement noise. Consider the following dynamics: 
	\begin{align}\label{eq: DT-system-process noise}
\begin{split}
			\mbf{x}[k+1]&=\mbf{A} \mbf{x}[k]+\mbf{B}_S\mbf{u}^*_S[k]+\sum_{j\in S^c}\mbf{b}_j{u}^*_j[k]+\mbf{w}[k]\\
			\mbf{y}[k]&=\mbf{C}	\mbf{x}[k]+\mbf{v}[k], 	     
		\end{split}
	\end{align}
	where the noise random vectors $\mbf{w}[k]\overset{iid}{\sim} \mathcal{N}(\mbf{0},\mbf{Q})$ and $\mbf{v}[k]\overset{iid}{\sim} \mathcal{N}(\mbf{0},\mbf{R})$, with $\mbf{Q}\succeq \mbf{0}$ and $\mbf{R} \succ \mbf{0}$, are uncorrelated. Let $\mbf{y}=[\mbf{y}[0]^\transpose\ldots \mbf{y}[N]^\transpose]^\transpose$, and from \eqref{eq: DT-system-process noise}, note that 
		\begin{align}
		\mbf{y}=&\mbf{O}\bsx^*_0+{\mbf{J}}_S\bsu^*_S+\sum_{j\in S^c}{\mbf{J}}_j\bsu^*_j+{\mbf{J}_{w}\mbf{w}+\mbf{v}},\label{eq: measurement model4}
	\end{align}
	where $\mbf{w}\triangleq[\mbf{w}[0]^\transpose\ldots \mbf{w}[N]^\transpose]^\transpose$ and $\mbf{v}=[\mbf{v}[0]^\transpose\ldots \mbf{v}[N]^\transpose]^\transpose$. The noise response matrix $\bf{J}_w$ is obtained by replacing $\mbf{H}_k^{(1)}$ in $\mbf{J}_1$, given by \eqref{eq: system matrices}, with $\mbf{C}\bfA^{k-1}$, for all $k\geq 0$. Because $\mbf{w}$ and $\mbf{v}$ are Gaussian, it follows that $\mbf{J}_{w}\mbf{w}+\mbf{v}\sim \mathcal{N}(\mbf{0},\boldsymbol{\Sigma}_{\wtilde{\mbf{v}}})$, where $\boldsymbol{\Sigma}_{\wtilde{\mbf{v}}}=[\mbf{J}_{w}\mbf{w}+\mbf{v}][\mbf{J}_{w}\mbf{w}+\mbf{v}]^\transpose$. Finally, define $\widetilde{\sigma}^2=\|\boldsymbol{\Sigma}_{\wtilde{\mbf{v}}}\|_2$. 
	
	Suppose that we solve the group LASSO problem in \eqref{eq: group LASSO opt} for the model in \eqref{eq: measurement model4}. Then, Theorems in \ref{thm: location-consistency} and \ref{thm: l2-consistency} hold true for $\sigma^2=\widetilde{\sigma}^2$. However, the modified noise variance ($\widetilde{\sigma}^2$) could be large depending on the system matrices in \eqref{eq: DT-system-process noise}.

\subsection{Mutual Incoherence: Frequency Domain}
Thus far we discussed the location recovery- and estimation-consistency of the group LASSO estimator in \eqref{eq: group LASSO opt} assuming that
assumptions in (A1)-(A3) hold of which the first two are satisfied by stable dynamical systems with $\sysmatS$ having no invariant zeros\footnote{Systems having invariant zeros lie in a zero measure set \cite{BDOA:MD-09}.}. However, (A3) might not hold for arbitrary systems, and moreover, verifying \eqref{eq: mutual incoherence} can be computationally demanding when either $N$ (the measurement horizon) or $n$ (dimension of system matrix $\mbf{A}$) is large. In what follows, we bound $\max_{j \in S^c}\|\mbf{J}_j^\transpose\bsPsi_S(\bsPsi_S^\transpose\bsPsi_S)^+\|_2$ in \eqref{eq: mutual incoherence} using a quantity that depends on the transfer function matrices associated with $(\mbf{A},\mbf{B}_S,\mbf{C})$ and $(\mbf{A},\mbf{b}_j,\mbf{C})$, for $j \in S^c$. The advantage is that this upper bound can computed efficiently, as it depends only on the lower dimensional system matrices but not on $N$. 

To simplify the exposition, we assume $\mbf{x}_0=\mbf{0}$; thus, $\bsPsi_S=\mbf{J}_S$. Similar to the transfer matrix $\mcal{G}_S[z]$ in \eqref{eq: TF matrix}, define $\mcal{G}_j[z]=\mbf{C}(z\mbf{I}-\bfA)^{-1}\mbf{b}_j$ and $\mcal{G}_{S^c}[z]=\mbf{C}(z\mbf{I}-\bfA)^{-1}\mbf{B}_{S^c}$, where $\mbf{B}_{S^c}$ is the matrix composed of columns $\mbf{b}_j$, with $j \in S^c$. 
\begin{theorem} \label{thm: frequency MIC} Assumption (A6) holds if $\mathrm{nRank} \mathcal{Z}_S\!=\!n\!+\!m^*$ and 
\begin{align}
				\max_{j \in S^c}\max_{\{z\in \mathbb{C}:|z|=1\}}\norm{\mathcal{G}^{+}_S[z]\mcal{G}_j[z]}_2&\leq \alpha/m^*<1. \label{eq: frequency mutual incoherence}
			\end{align}
\end{theorem}
\begin{proof}
    See Appendix.
\end{proof}
We refer to the expression in \eqref{eq: frequency mutual incoherence} as the frequency domain mutual incoherence condition. 
Thus to verify Assumption (A6), we need to check if the worst case gain of the transfer matrix $\mathcal{G}^{+}_S[z]\mcal{G}_j[z]$ is bounded above by $\alpha/m^*$; see Fig.~\ref{fig: MIC plot}. If computing \eqref{eq: frequency mutual incoherence} is prohibitive for each $j \in S^c$, we can resort to the weaker condition: $\max_{\{z\in \mathbb{C}:|z|=1\}}\norm{\mathcal{G}^{+}_S[z]\mcal{G}_{S^c}[z]}_2\leq \alpha/m^*<1$. To appreciate the condition in \eqref{eq: frequency mutual incoherence}, we take $\mathcal{Z}$-transform of system in \eqref{eq: DT-system-state}-\eqref{eq: DT-system-measurements}, and then note that 
\begin{align*}
    \mbf{y}[z]=\mcal{G}_S[z]\mbf{u}_S[z]+\sum_{j\in S^c}\mcal{G}_j[z]\mbf{u}_j[z], \quad \forall z \notin \text{spec}(\mbf{A}). 
\end{align*}
By pre-multiplying the above identity with $\mcal{G}^{+}_S[z]$, we have
\begin{align*}
    \mcal{G}^{+}_S[z]\mbf{y}[z]&=\mbf{u}_S[z]+\sum_{j\in S^c}\mcal{G}^{+}_S[z]\mcal{G}_j[z]\mbf{u}_j[z] \quad \forall z \notin \text{spec}(\mbf{A})\\
    &=\mbf{u}_S[z]+\mcal{G}^{+}_S[z]\mcal{G}_{S^c}[z]\mbf{u}_{S^c}[z],
\end{align*}
Thus to recover $\mbf{u}_S[z]$ accurately, the gain $\|\mathcal{G}^{+}_S[z]\mcal{G}_j[z]\|_2$ or $\|\mathcal{G}^{+}_S[z]\mcal{G}_{S^c}[z]\|_2$ needs to be small. 

\begin{figure}[!t]
	\centering\includegraphics[width=1.0\linewidth]{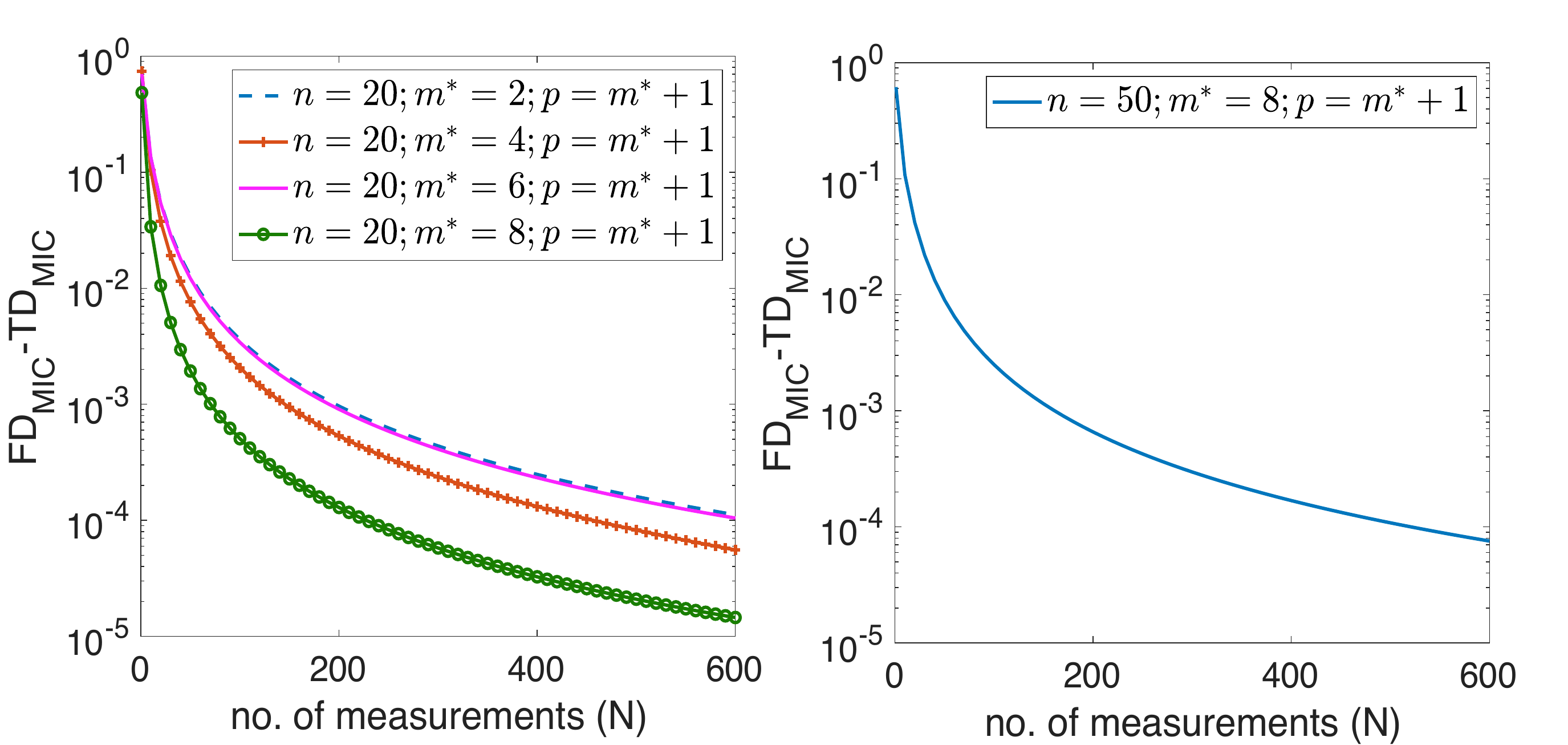}
	\caption{\footnotesize Illustration of Theorem \ref{thm: frequency MIC} for system matrices generated using \texttt{drss} command in MATLAB. The number of possible sources $m=10$. In both panels, the y-axis, $FD_\text{MIC}-TD_\text{MIC}$, is the error between frequency- and time-domain mutual incoherences. (Left panel) We fix $n=20$ and plot $FD_\text{MIC}-TD_\text{MIC}$ for several values of $m^*$. (Right panel) For a large dimensional matrix $\mbf{A}$, we fix $m^*$ and $p$, and plot $FD_\text{MIC}-TD_\text{MIC}$ for several values of system dimension $n$.  In both the panels, the error is positive and is monotone in $N$ implying that $FD_\text{MIC}\geq TD_\text{MIC}$, as predicted by Theorem \ref{thm: frequency MIC}.}
	\label{fig: MIC plot}
\end{figure}

We conclude this section by highlighting three special cases where \eqref{eq: frequency mutual incoherence} holds: (i) $\mathcal{R}(\mathcal{G}_{S^c}[z])\subseteq \mathcal{R}^{\perp}(\mathcal{G}^+_{S}[z])=\mathcal{R}^{\perp}(\mathcal{G}^T_{S}[z])$, for all $|z|=1$. In other words, the columns of $\mathcal{G}_{S^c}[z]$ lie in the left nullspace of $\mathcal{G}_{S}[z]$; (ii) $\mathcal{G}[z]=[\mathcal{G}_S[z]\,\,\mathcal{G}_{S^c}[z]]$ is all-pass\footnote{A real rational transfer function matrix $G[z]$ is all-pass if $\mathcal{G}[z]\mathcal{G}[1/z]=\mbf{I}$.}; and (iii) $\mcal{G}_{S^c}[z]=\alpha\mcal{G}_{S}[z]$. The first two cases are rather strong and does not allow columns of $\mcal{G}_{S^c}[z]$ to be in the range space of $\mcal{G}_{S}[z]$. Instead, (ii) models another extreme where the range spaces of $\mcal{G}_{S}[z]$ and $\mcal{G}_{S^c}[z]$ are aligned with each, modulo the factor $\alpha \in [0,1)$. The latter case in the compressed sensing literature is referred to as overcomplete dictionaries \cite{JJF:04}.

	\section{Simulations}
	{We illustrate the performance of the group LASSO estimator on a large-scale power network and a random system. The following proposition states that the unknown input and initial state can be estimated in two stages. Consequently, we use off-the-shelf ADMM \cite{SELP09} to estimate the input first and then use this estimate to compute the initial state.} 
	\begin{proposition}\label{prop: decoupled estimator}
	Suppose that system in \eqref{eq: DT-system-state}-\eqref{eq: DT-system-measurements} is observable. The optimization problem \eqref{eq: group LASSO opt} is equivalent to 
		\begin{align}
			\bshu&=\argmin_{\substack{\bsu \in \real^{mT}}}\frac{1}{2T}\norm{\bs{\Pi}(\mbf{y}-\mbf{J}\bsu)}_2^2+\lambda_T \sum_{j=1}^m\norm{\bsu_j}_2, \label{eq: u_hat sub problem} \\
			\bshx_0&=\mbf{O}^+(\mbf{y}-\mbf{J}\bshu),  \label{eq: x_hat sub problem}
		\end{align}
		where $\mbf{O}^+=(\mbf{O}^\transpose\mbf{O})^{-1}\mbf{O}^\transpose$ and $\bs{\Pi}=\mbf{I}-\mbf{O}\mbf{O}^+$. 
	\end{proposition}

	The proof follows from the KKT conditions \cite{NS:JF:TH:RT:13}. The inputs to the ADMM \cite{SELP09} are the system matrices $(\mbf{A},\mbf{B}, \mbf{C})$, the measurement $\mbf{y}$, and the tuning parameter $\lambda_T\geq 0$. Finally, we note that the two-stage estimation method is one way to implement the group LASSO numerically. One may also use other numerical algorithms to estimate $(\mbf{x}_0^*, \mbf{u}^*)$ in one shot. 
	
	We evaluate the group LASSO estimator's localization performance using the false-positive rate (FPR):= $|S^c\cap \what{S}|/|S^c|$, the false-negative rate (FNR):= $|S\cap \what{S}^c|/|S|$, and the exact recovery rate (ERR):= $(|S\cap \what{S}|+|S^c\cap \what{S}^c|)/m$. Recall that FPR and FNR, respectively, measure the proportion of inputs that are falsely identified and left out. Instead, we quantify the estimation performance using the error metrics: $\|\mbf{x}_0^*-\what{\mbf{x}}_0\|_2/\|\mbf{x}_0^*\|_2$ and $\|\mbf{u}^*-\what{\mbf{u}}\|_2/\|\mbf{u}\|_2$. 
	For the test cases below, the results are averaged over 50 runs. 

\smallskip 
		\textit{(Power system)} We apply our estimator in \eqref{eq: u_hat sub problem} to localize the sources of forced oscillatory (FO) inputs in the IEEE 68 bus system 16 machine system (see Fig.~\ref{fig: IEEE_68bus_system}). Each machine (or generator) consists of ten states, including rotor angle, speed, and the states of the AVR (automatic voltage regulator) and PSS. We model FOs as inputs injected by the AVRs and use bus voltage magnitudes as measurements.  For the sampling time $\delta t=0.1$, we obtained the system matrices $\mbf{A} \in \mathbb{R}^{160\times 160}$,  $\mbf{B}\in \mathbb{R}^{160\times 16}$, and $\mbf{C}\in \mathbb{R}^{p\times 160}$, where $p\leq 68$, using the Power System Toolbox \cite{JHC:KWC-92}. Among $m=16$ possible inputs, we assume $m^*=3$ with the following inputs: 
		$u^*_1[k]=0.5\sin[(2\pi f\delta t) \,k]+w[k]$, $u^*_6[k]=0.6\sin[(2\pi f\delta t) \,k]+w[k]$, and $u^*_{13}[k]=0.7\sin[(2\pi f\delta t) \,k]+w[k]$, where $f=1.5\,\mathcal{U}(0,1)$ and $w[k]\sim \mathcal{N}(0,0.05^2)$. We set $p=4$ and choose sensor locations arbitrarily with the only exception that these are non-collocated with inputs (shown in Fig.~\ref{fig: IEEE_68bus_system}). Let $\mbf{x}_0=\mbf{0}$ (the non-zero case is considered in the subsequent case). Finally, we let $N=100$ and the noise variance $\sigma^2=0.01$. 
		
		In Fig.~\ref{fig: FPR_68bus_system}, we plot the FPR, FNR, and ERR with respect to $\lambda_T$. As expected, the FNR increases with $\lambda_T$ whereas the FPR decreases with $\lambda_T$, although not monotonically. From the bottom left panel, we can infer that values of $\lambda_T\in (0.3,\, 0.4)$ yield maximum ERR. In the bottom right panel, note that for $\lambda_T=0.288$, the group LASSO estimator accurately localized inputs among 40 out of 50 runs. In Fig.~\ref{fig: waveforms_68bus_system}, for a measurement realization where the group LASSO estimator identified true locations, we plot the inputs estimated by the group LASSO and the reduced model based OLS estimators.

  	\begin{figure}[!t]
 	\centering\includegraphics[width=0.8\linewidth]{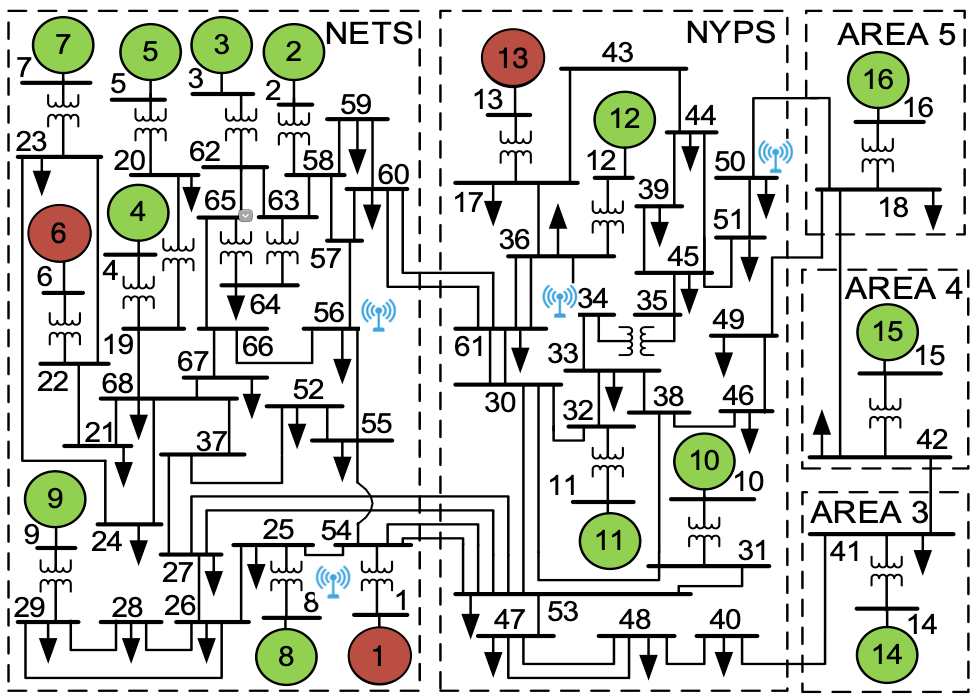}
  	\caption{\scriptsize IEEE 16 machine 68 bus system \cite{68B16M}. Circles, arrows, and curly windings, respectively, denote generator buses, load buses, and transformers. The FO input enters through set points of AVRs associated with the generators at buses $\{1, 6, 13\}$ (red circles). Sensors are located at buses $\{8, 34, 50, 56\}$.}
  	\label{fig: IEEE_68bus_system}
  \end{figure}
  
    \begin{figure}[!t]
 	\centering\includegraphics[width=0.97\linewidth]{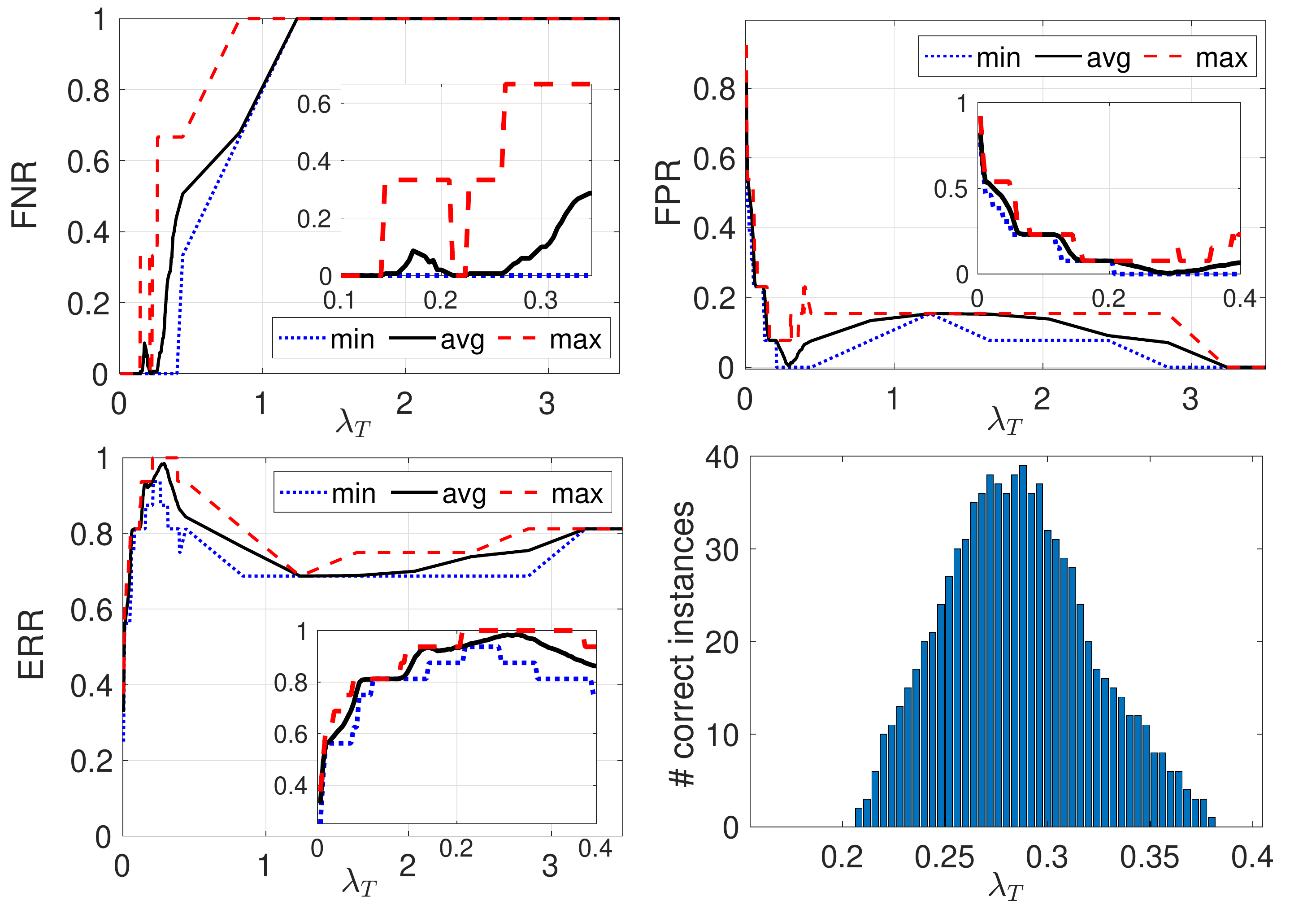}
  	\caption{\scriptsize False negative rate (FNR), false positive rate (FPR), exact recovery rate (ERR), and the number of exactly recovered instances among 50 runs of the IEEE 16 machine 68 bus system data using the group LASSO.}
  	\label{fig: FPR_68bus_system}
  \end{figure}
    \begin{figure}
 	\centering\includegraphics[width=0.97\linewidth]{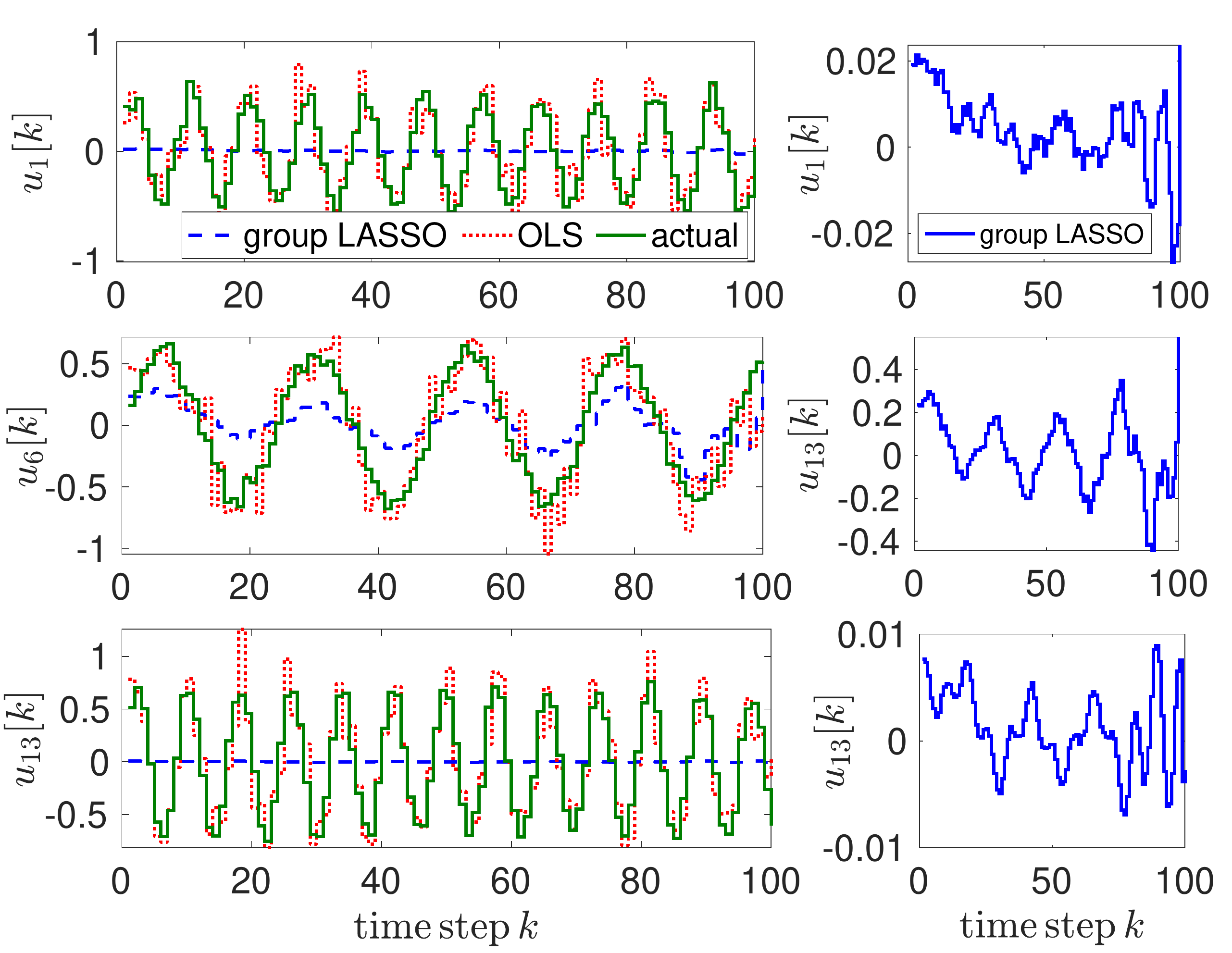}
  	\caption{\scriptsize FO inputs recovered by the group LASSO and OLS estimators. We used \eqref{eq: OLS estimate} to compute the OLS estimate using the locations recovered by the group LASSO. (Left panel) As predicted by Theorem \ref{thm: l2-consistency}, the OLS provides a better estimate than the LASSO estimator. (Right panel) zoomed plot of the group LASSO estimate.}
  	\label{fig: waveforms_68bus_system}
  \end{figure}


	\textit{(Large-scale random system)}
    Following \cite{SMF-FG-SG-AYK-DS:19}, we generate matrices as follows: $\mbf{A}_{ij}\overset{\text{iid}}{\sim} \mathcal{N}(0,1/{n})$; $\mathbf{C}_{ij}\overset{\text{iid}}{\sim} \mathcal{N}(0,1)$; and $\mathbf{B}^\transpose=\begin{bmatrix}
        \mbf{I}_m^\transpose&
        \mbf{0}^\transpose 
    \end{bmatrix}$. We let $\mbf{x}_0\sim \mathcal{N}(\mbf{0},\mbf{I}_n)$ and the measurement noise variance parameter $\sigma=0.01$. We set $n=50$, $m=30$, and $m^*=5$. The active set $S=\{1,2,3,4,5\}$ and ${u}_j[k]$ is sampled uniformly on $[-2,2]$, for all $j \in S$ and $k \in [N]$. The sensors measures the first $p (\le n)$ states. In Fig.~\ref{fig: rand_sys_error}, for $p=15$, we plot the average estimation error metrics as a function of the measurement horizon ($N$). In both the panels, estimation errors remain uniform across $N$ because the number of (to be estimated) inputs also increase with $N$. Given the relation in \eqref{eq: x_hat sub problem}, the estimation error of $\mbf{x}^*_0$ is slightly higher than that of the unknown input. Finally, for greater estimation accuracy, one can always use the reduced model-based OLS estimator.
	
	In Fig.~\ref{fig: MIC_rand_systems}, we show the average mutual incoherence (MIC) in \eqref{eq: mutual incoherence} as a function of $p$, for two cases: $\mbf{x}_0=\mbf{0}$ and $\mbf{x}_0\ne \mbf{0}$. We computed both $\ell_1$- and $\ell_2$-norm based MICs. As pointed out in Section \ref{sec: main results}, and confirmed by our plots in the left panel of Fig.~\ref{fig: MIC_rand_systems}, $\ell_2$-norm based MIC assumption is stronger than the $\ell_1$-norm. Further, when $\mbf{x}_0= \mbf{0}$, the MIC is satisfied (that is, less than one) for as few as $p=6$ sensors. Here, $p=m^*+1$. Instead, when $\mbf{x}_0\ne \mbf{0}$, we need at least $p=18$ sensors to ensure that MIC is below one. Given $m^*$, theoretical relationships between $p$ and MIC is left for future research. 
	    \begin{figure}[!t]
		\centering\includegraphics[width=1.0\linewidth]{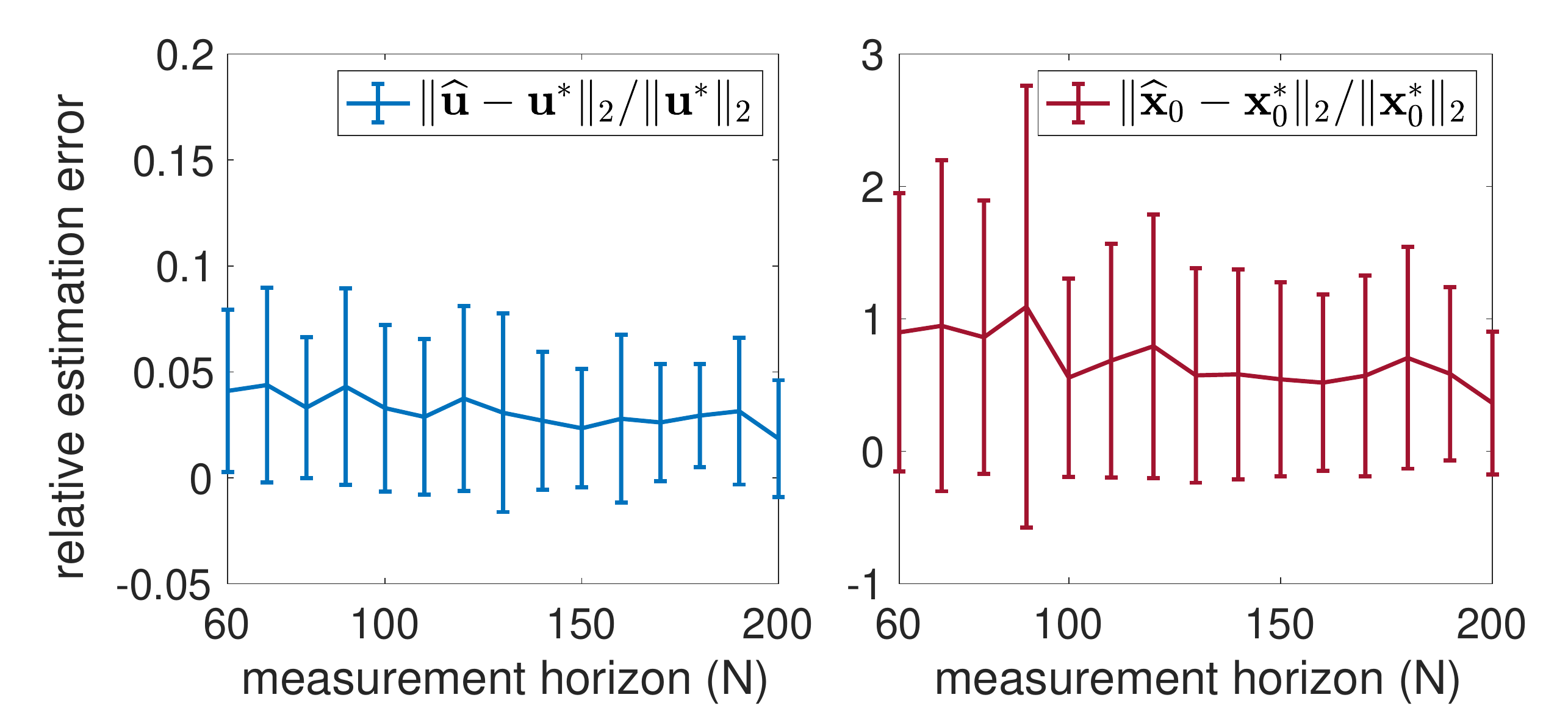}
		\caption{\scriptsize Relative estimation error. Left panel: unknown inputs. Right panel: initial state.}
		\label{fig: rand_sys_error}
	\end{figure}
		    \begin{figure}[!t]
	\centering\includegraphics[width=1.0\linewidth]{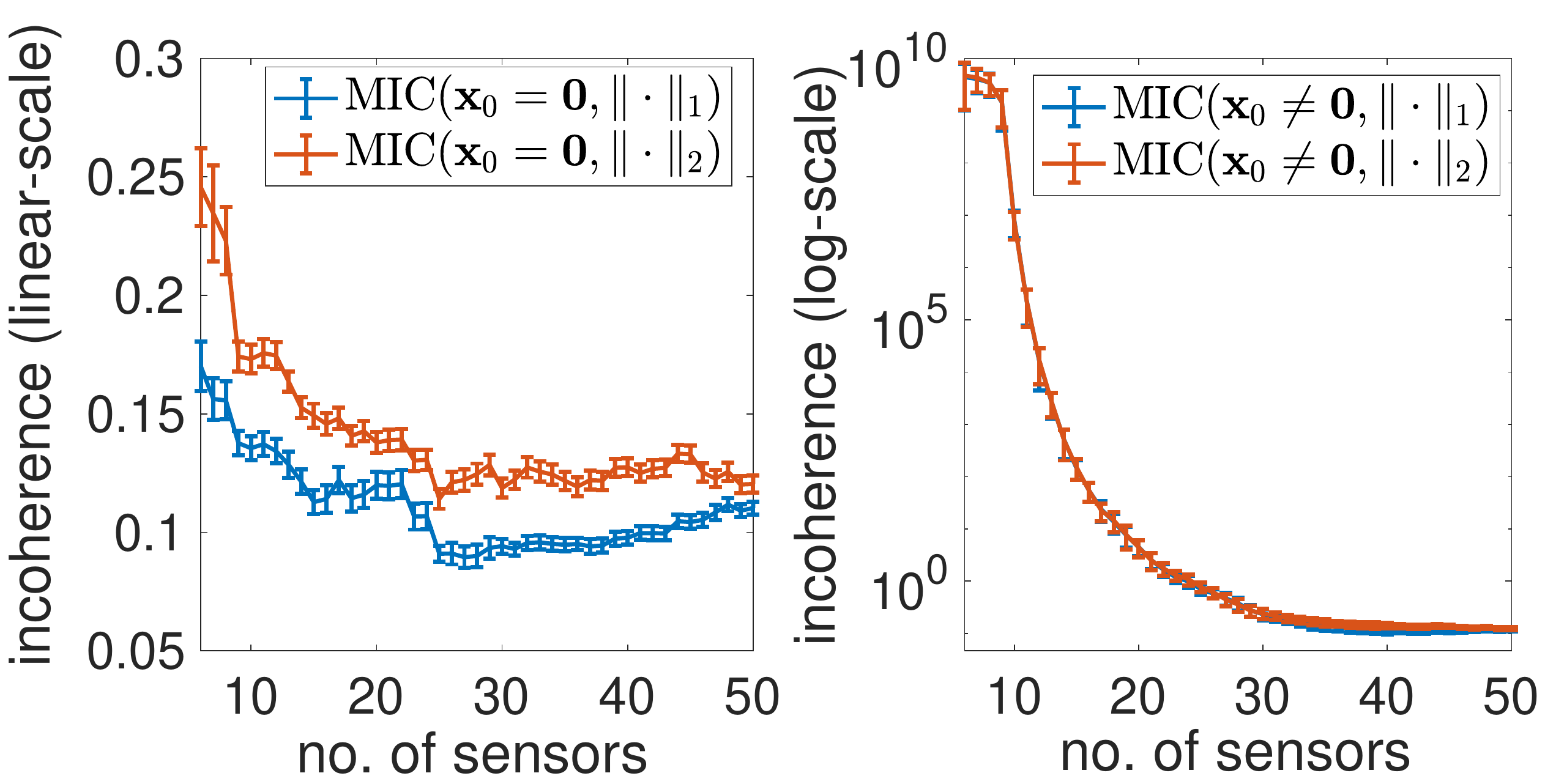}
	\caption{\scriptsize Mutual incoherence vs. number of sensors. Left: $\mbf{x}^*_0=\mbf{0}$. Right: $\mbf{x}^*_0\ne \mbf{0}$}
	\label{fig: MIC_rand_systems}
\end{figure}

\vspace{-2.0mm}
\section{Conclusions}
We study a group LASSO estimator for locating the sources of unknown forced inputs as well as estimating these inputs along with the initial state from noisy measurements. We derive sufficient conditions under which the group LASSO estimate is location- and estimation-recovery consistent. In doing so, we have extended the existing theory of the group LASSO estimator for static regression models to 
linear models generated by $d$-delay (left) invertible linear dynamical systems with unknown initial state. Our results establish a relationship between time- and frequency-domain mutual incoherence conditions. The latter condition provides insight into the structural aspects of transfer matrices associated with the zero and non-zero inputs. Finally, we have validated the performance of our proposed group LASSO estimator via simulations. 

Possible future work includes extending the group LASSO framework for linear and non-linear systems with additive state noise. As pointed out in Section \ref{sec: noisy dynamics}, in the presence of state noise, whitening the measurements can be detrimental to the performance of group LASSO. Two directions seem promising to tackle this issue: (i) to use the puffer-transformation method in \cite{JJ-KR:09} and (ii) to consider the group LASSO estimator for predictor or innovation form of the state-space model.

\ifCLASSOPTIONcaptionsoff
\newpageor
\fi

\bibliographystyle{unsrt}
\bibliography{BIB}


\section{Appendix}
\subsection{KKT conditions and PDW Construction}


\begin{proposition}{\bf \emph{(Karush-Kuhn-Tucker (KKT) conditions)}}\label{prop: KKT conditions}
	A necessary and sufficient condition for $(\what{{\bsx}}_0,\bshu)$, with $\bshu^\transpose=[\bshu_1^\transpose,\ldots,\bshu_m^\transpose]$, to be a solution of \eqref{eq: group LASSO opt} is 
		\begin{align}
			-\frac{1}{T}\mbf{O}^\transpose[\bsy - \bsO\bshx_0-\sum_{j=1}^m\bsJ_j\bshu_j]&=\bfzeros\label{eq: KKT 1a}\\ 
			-\frac{1}{T}\mbf{J}_i^\transpose[\bsy - \bsO\bshx_0-\sum_{j=1}^m\bsJ_j\bshu_j]+\lambda_T\bshz_j&=\bfzeros\label{eq: KKT 2a}
		\end{align}
	for all $j \in \{1,\ldots,m\}$. Here, $\bshz_j$ is the subgradient of $\norm{\bshu_j}_2$; that is, $\bshz_j={\bshu_j}/{\norm{\bshu_j}_2}$ if $\bshu_j\ne \bfzeros$, else $\bshz_j\in \{\mbf{q}: \norm{\mbf{q}}_2\leq 1\}$.
\end{proposition}


\newcommand{\bsteta}{\boldsymbol{\beta}}
\newcommand{\bshteta}{\widehat{\boldsymbol{\beta}}} 
\newcommand{\bshtetatil}{\widehat{\wtilde{\bsbeta}}} 
\newcommand{\bshutil}{\widehat{\wtilde{\mbf{u}}}} 
\newcommand{\bstetatil}{{\wtilde{\bsbeta}}} 
\newcommand{\Sc}{{S^c}}

The proof is given in \cite{NS:JF:TH:RT:13}. Without loss of generality let $S=\{1,\ldots,m^*\}$ and $S^c=\{m^*+1,\ldots,m\}$. Let $\what{\mbf{u}}_S^\transpose[k]=[\what{\mbf{u}}_{1}[k]\ldots,\what{\mbf{u}}_{m^*}[k]]$, for all $k \in \{0,\ldots,N\}$, where $\what{\mbf{u}}_{j}[k]$ is the $k$-th entry of $\what{\mbf{u}}_{j}$. Define $\what{\mbf{u}}_S^\transpose=[\what{\mbf{u}}_S^\transpose[0],\ldots,\what{\mbf{u}}_S^\transpose[N]]$.Thus, 
\begin{align}
    [\what{\mbf{u}}_1^\transpose,\ldots,\what{\mbf{u}}_{m^*}^\transpose]^\transpose=\mathbf{P}\what{\mbf{u}}_S,  
\end{align}
for some permutation matrix $\mbf{P}$. Further, we can verify that $[\mbf{J}_1\ldots, \mbf{J}_{m^*}]\mbf{P}=\mbf{J}_S$ (as in \eqref{eq: system matrices1}). Let $\what{\bsbeta}_S\triangleq [\what{\mbf{x}}_0^\transpose \,\, \what{\mbf{u}}_S^\transpose]^\transpose$ Then, 
 \begin{align}\label{eq: hatted identity J}
\mbf{O}\what{\mbf{x}}_0+\sum_{j \in S}\mbf{J}_j\what{\mbf{u}}_j
     &=\underbrace{\begin{bmatrix}
 	                \mbf{O} & \mbf{J}_S
 	            \end{bmatrix}}_{\bs{\Psi}_S}{\what{\bsbeta}_S}
\end{align}
Using these facts, expressions in \eqref{eq: KKT 1a}-\eqref{eq: KKT 2a} can be written as 
\begin{align}\label{eq: KKT partitioned}
\hspace{-3.0mm}	-\frac{1}{T}\begin{bmatrix}
		\bsPsi_S^\transpose\\
		\widetilde{\bfJ}_{S^c}^\transpose
	\end{bmatrix}[\bsy - \bsO\bshx_0-\sum_{j=1}^m\bsJ_j\bshu_j]\!+\!\lambda_T\begin{bmatrix}
		\bfzeros\\
				\mbf{P}^\transpose \bshz_S\\
		\hline 
\bshz_{S^c}
	\end{bmatrix}\!=\!\begin{bmatrix}
		\bfzeros\\
		\bfzeros\\
		\hline 
		\bfzeros
	\end{bmatrix},
\end{align}
where $\widetilde{\bfJ}_{S^c}=[\mbf{J}_{m*+1},\ldots,\mbf{J}_{m}]$, $\bshz_S^\transpose=[\bshz_{1}^\transpose,\ldots,\bshz_{m^*}^\transpose]$, and $\bshz_{S^c}^\transpose=[\bshz_{m*+1}^\transpose,\ldots,\bshz_{m}^\transpose]$.

\medskip 
\noindent{\bf Primal-dual witness (PDW) construction}: We prove Theorems \ref{thm: l2-consistency} and \ref{thm: location-consistency} using the PDW construction technique\footnote{The PDW construction is not an algorithm for solving the group LASSO problem in \eqref{eq: group LASSO opt}:  This is because to solve the sub-problem in step (b) of PDW, we need to know the active set $S$. However, PDW construction technique helps to prove consistency results for the LASSO type problems.} \cite{MJW:19}: 


\begin{enumerate}[label=(\alph*)]
	\item Set $\bshu_{j}=0$, for all $j \in S^c$. 
	\item Let $(\what{\bfx}_0,\what{\bsu}_{1}\ldots,\what{\bsu}_{m^*})$ be the solution of the sub-problem: 
	\begin{align}\label{eq: oracle sub-problem}
\hspace{-6.0mm}\min_{\substack{\mbf{x}_0;\\ \bfu_1,\ldots,\bfu_{m^*}}}\frac{1}{2T}{\norm{\mbf{y}\!-\!\mbf{O}\mbf{x}_0\!-\!\sum_{j=1}^{m^*}\mbf{J}_j\mbf{u}_j}_2^2}\!+\!\lambda_{T}\sum_{j=1}^{m^*}\|\mbf{u}_j\|_2.
	\end{align}
	Choose the sub-gradient $\bshz_S=[\bshz_1^\transpose,\ldots,\bshz_{m^*}^\transpose]^\transpose$
 such that 
	\begin{equation}
		-\frac{1}{T}\bsPsi_S^\transpose\left[\bsy- \bsO\bshx_0-\textstyle\sum_{j=1}^{m^*}\bsJ_j\bshu_j\right]+\lambda_T\begin{bmatrix}
			\bfzeros\\
			\mbf{P}^\transpose \bshz_S
		\end{bmatrix}=\bfzeros. 
	\end{equation}

	\item Solve $\bshz_\Sc=[\bshz_{m^*+1}^\transpose,\ldots,\bshz_{m}^\transpose]^\transpose$ using \eqref{eq: KKT partitioned}, and check if $\|\bshz_j\|_2\leq 1$, for all $j \in S^c=\{m^*+1,\ldots,m\}$. 
\end{enumerate}

By construction, $(\what{\mbf{x}}_0,\what{\mbf{u}}_1\ldots,\what{\mbf{u}}_{m^*})$, $\bshz_S$, and $\bshz_\Sc$ that we determined in steps (a), and (b) satisfy conditions in \eqref{eq: KKT partitioned}. The PDW construction is said to be successful if $\bshz_{S^c}$ satisfies the strict dual feasibility condition: $\|\bshz_j\|_2\leq 1$, for all $j \in S^c$. 

\vspace{-3.0mm}
\subsection{Proofs of Theorems in Section \ref{sec: main results}}

For the estimate in \eqref{eq: oracle sub-problem}, define
\begin{align}\label{eq: PDW estimate}
    \what{\bsbeta}_\text{PDW}=(\what{\mbf{x}}_{0},\what{\bsu}_1,\ldots,\what{\bsu}_{m^*},\underbrace{\mbf{0}_{(N+1)},\ldots,\mbf{0}_{(N+1)}}_{m-m^*}). 
\end{align}

\begin{lemma}\label{lma: PDW success}
	 Suppose that the PDW construction succeeds. If delay $d>0$, $\what{\bsbeta}=\what{\bsbeta}_\mathrm{PDW}$ is an optimal solution of \eqref{eq: group LASSO opt}. If $d=0$, $\what{\bsbeta}=\what{\bsbeta}_\mathrm{PDW}$ is the "unique" optimal solution. 
\end{lemma} 
\begin{proof}
We follow the proof technique in \cite[Lemma 7.23]{MJW:19}. Let $d\geq 0$. Because the PDW construction succeeds, $\what{\bsbeta}_\mathrm{PDW}$ is an optimal solution of \eqref{eq: group LASSO opt} with subgradient vector $\bshz^\transpose=[\bshz_0^\transpose, \bshz_1^\transpose,\ldots,\bshz_{m}^\transpose]$ satisfying $\bshz_0=\bfzeros$, $\|\bshz_j\|_2=1$ for $j \in S$, and $\|\bshz_j\|_2<1$ for $j \in S^c$. We now show that any optimal solution of \eqref{eq: group LASSO opt} is supported on the set $S$. 

With a slight abuse of notation, let $\mbf{u}^\transpose=[\mbf{x}_0^\transpose, \mbf{u}^\transpose_1,\ldots,\mbf{u}^\transpose_m]$ and denote $F(\mbf{u})=\frac{1}{2T}\|\mbf{y}-\mbf{O}\mbf{x}_0+\sum_{j=1}^m\mbf{J}_j\mbf{u}_j\|_2^2$. Let $\nabla F(\mbf{u})$ be the gradient of $F(\mbf{u})$ with respect to $\mbf{u}$. Then, for any other optimal solution $\wtilde{\mbf{u}}$ of \eqref{eq: group LASSO opt}, we have $F(\what{\mbf{u}})+\lambda_T\bshz^\transpose\bshu=F(\wtilde{\mbf{u}})+\lambda_T\sum_{j = 1}^m\|\wtilde{\bsu}_j\|_2$
The last equality follows because $\sum_{j = 1}^m\bshz_j^\transpose\bshu_j=\sum_{j=1}^m\|\bshu_j\|_2$. 
Hence, $F(\what{\mbf{u}})-\lambda_T\bshz^\transpose(\wtilde{\bfu}-\bshu)=F(\wtilde{\mbf{u}})+\lambda_T\sum_{j = 1}^m\|\wtilde{\bsu}_j\|_2-\lambda_T\bshz^\transpose\wtilde{\bfu}$. Instead, from the zero-subgradient conditions in \eqref{eq: KKT 1a}-\eqref{eq: KKT 2a}, we have $\lambda_T\bshz=-\nabla F(\bshu)$. Putting the pieces together, we have
\begin{align*}
    F(\what{\mbf{u}})+\nabla F(\bshu)^\transpose(\wtilde{\bfu}-\bshu)-F(\wtilde{\mbf{u}})=\lambda_T(\sum_{j = 1}^m\|\wtilde{\bsu}_j\|_2-\bshz^\transpose\wtilde{\bfu}).
\end{align*}
By convexity of $F$, the left-hand side is negative. As a result, $\sum_{j = 1}^m\|\wtilde{\bsu}_j\|_2\leq \bshz^\transpose\wtilde{\bfu}=\sum_{j=1}^m\bshz^\transpose_j\wtilde{\mbf{u}}_j$, where $\bshz_0=\bfzeros$. Since we also have $\sum_{j=1}^m\bshz^\transpose_j\wtilde{\mbf{u}}_j\leq \sum_{j=1}^m\|\bshz_j\|_2\|\wtilde{\mbf{u}}_j\|_2\leq\sum_{j=1}^m\|\wtilde{\mbf{u}}_j\|_2$, we must have $\sum_{j = 1}^m\|\wtilde{\bsu}_j\|_2\!=\!\sum_{j=1}^m\bshz^\transpose_j\wtilde{\mbf{u}}_j$. Because $\|\bshz_j\|_2<1$ for $j \in S^c$, the above equality can only occur if $\wtilde{\bsu}_j=\mbf{0}$, for all $j \in S^c$. To see this notice that $\sum_{j=1}^m\bshz^\transpose_j\wtilde{\mbf{u}}_j\!=\!\sum_{j \in S}\bshz^\transpose_j\wtilde{\mbf{u}}_j+\sum_{j \in  S^c}\|\bshz_j\|_2\|\wtilde{\bsu}_j\|_2\cos(\theta_j)$, 
where $\theta_j$ is the angle between $\bshz_j$ and $\widetilde{\bfu}_j$, and $\|\bshz_j\|_2\cos(\theta_j)\in (-1,1)$. Thus, all optimal solutions $\what{\bsbeta}$ are such that $\what{\bsbeta}_j=\mbf{0}$ for all $j \in S^c$. These solutions can be obtained by solving \eqref{eq: oracle sub-problem}. Finally, for $d=0$, the assumption in (A2) ensures that \eqref{eq: oracle sub-problem} is strictly convex, and hence, $\what{\bsbeta}$
is a unique minimizer. \end{proof}

\noindent{\textit{Proof of Theorem \ref{thm: location-consistency}}}: Suppose the PDW construction succeeds. The proof of part (a) is given in Lemma \ref{lma: PDW success}. Further, in view of Lemma \ref{lma: PDW success}, $\what{\bsbeta}=\what{\bsbeta}_\mathrm{PDW}$ is an optimal solution of \eqref{eq: group LASSO opt}. Thus, all the optimal input vectors are supported on the set $S$, i.e., $\what{S}\subset S$, where $\what{S}=\{j:\what{\mbf{u}}_j\ne 0\}$; thus, part (b) holds. 

We show that the PDW construction succeeds with probability at least $1-2\exp(-T\delta^2/2)$ by showing that $\|\bshz_j\|_2\leq 1$, for all $j \in S^c$. Here, $\bshz_{j}$ is determined in the step (c) of PDW construction. Let $\what{\bsbeta}_S$ be as in \eqref{eq: hatted identity J}. By substituting $\mbf{y}$ (given in \eqref{eq: measurement model1*}) and $\what{\mbf{u}}_\Sc=\mbf{0}$ in \eqref{eq: KKT partitioned}, we obtain 
\begin{align}\label{eq: block sub-gradient1}
	\frac{1}{T}\begin{bmatrix}
		\bsPsi_S^\transpose\bsPsi_S & \bsPsi_S^\transpose\bsJtil_{S^c}\\
		\bsJtil_{S^c}^\transpose\bsPsi_S & \bsJtil_{S^c}^\transpose\bsJ_{S^c}
	\end{bmatrix}
	\begin{bmatrix}
		\bsbeta^*_S-\what{\bsbeta}_S\\
		\bfzeros 
	\end{bmatrix}\!+\!\frac{1}{T}
	\begin{bmatrix}
		\bsPsi_S^\transpose\\
		\bsJtil_{S^c}^\transpose
	\end{bmatrix}\mbf{v}\!=\!\lambda_T\begin{bmatrix}
		\bfzeros\\
		\mbf{P}^\transpose\bshz_{S}\\
		\hline 
		\bshz_\Sc
	\end{bmatrix}. 
\end{align}
Using the second block equation of \eqref{eq: block sub-gradient1}, solve for $\bshz_\Sc$ as 
\begin{align}\label{eq: zsc1}
	\bshz_\Sc&=\bsJtil_{S^c}^\transpose\bsPsi_S\left[\frac{\bsPsi_S^\dagger\bsPsi_S}{\lambda_T T}(\bsteta^*_S-\bshteta_S)\right]+\bsJtil_{S^c}^\transpose\left(\frac{\mbf{v}}{\lambda_TT}\right), 
\end{align}
where we used the fact $\bsPsi_S=\bsPsi_S\bsPsi_S^\dagger \bsPsi_S$. On the other hand, from the top block equation in \eqref{eq: block sub-gradient1}, we have 
\begin{align}\label{eq: difference theta1}
	\frac{1}{T}\bsPsi_S^\transpose\bsPsi_S(\bsteta^*_S-\bshteta_S)+\frac{1}{T}\bsPsi_S^\transpose\mbf{v}=\lambda_T\begin{bmatrix}
		\bfzeros\\
		\mbf{P}^\transpose\bshz_S
	\end{bmatrix}. 
\end{align}
Pre-multiply both sides of the equality in \eqref{eq: difference theta1} with $(\bsPsi_S^\transpose\bsPsi_S)^\dagger$. Then, use  the identities  $(\bsPsi_S^\transpose\bsPsi_S)^\dagger(\bsPsi_S^\transpose\bsPsi_S)^\dagger=\bsPsi_S^\dagger\bsPsi_S$ and 
$\bsPsi_S^\dagger=(\bsPsi_S^\transpose\bsPsi_S)^\dagger\bsPsi_S^\transpose$ (see \cite{AB-BS:03}) to get the following: 
\begin{align}\label{eq: difference theta2}
\hspace{-2.5mm}	{\bsPsi_S^\dagger\bsPsi_S}(\bsteta^*_S-\bshteta_S)=-\bsPsi_S^\dagger\mbf{v}+T\lambda_T(\bsPsi_S^\transpose\bsPsi_S)^\dagger\begin{bmatrix}
		\bfzeros\\
		\mbf{P}^\transpose\bshz_S
	\end{bmatrix}. 
\end{align}
Let $\boldsymbol{\Gamma}_S=[\mbf{I}-(\bsPsi_S\bsPsi_S^\dagger)]$. By substituting \eqref{eq: difference theta2} in the first term of the second equality in \eqref{eq: zsc1}, we can simplify $\what{\mbf{z}}_{S^c}$ as
\begin{align}\label{eq: zsc12}
	\bshz_\Sc&=\bsJtil_{S^c}^\transpose(\bsPsi_S^\dagger)^\transpose\begin{bmatrix}
		\bfzeros\\
		\mbf{P}^\transpose\bshz_S
	\end{bmatrix}+\bsJtil_{S^c}^\transpose\boldsymbol{\Gamma}_S\left(\frac{\mbf{v}}{\lambda_TT}\right), 
\end{align}
where we used the fact $(\bsPsi_S^\dagger)^\transpose=\bsPsi_S(\bsPsi_S^\transpose\bsPsi_S)^\dagger$. Thus,
\begin{align}\label{eq: zsc2}
	\bshz_j&=\bsJ_{j}^\transpose(\bsPsi_S^\dagger)^\transpose\begin{bmatrix}
		\bfzeros\\
		\mbf{P}^\transpose\bshz_S
	\end{bmatrix}+\bsJ_{j}^\transpose\boldsymbol{\Gamma}_S\left(\frac{\mbf{v}}{\lambda_TT}\right), \quad \forall j \in S^c
\end{align}
By the sub-multiplicative property of norms, for any $j \in S^c$,
\begin{align*}
    \norm{\bsJ_{j}^\transpose(\bsPsi_S^+)^\transpose\begin{bmatrix}
		\bfzeros\\
		\mbf{P}^\transpose\bshz_S
	\end{bmatrix}}_2&\leq \max_{j \in S^c}\|\bsJ_{j}^\transpose(\bsPsi_S^+)^\transpose\|_2\norm{\begin{bmatrix}
				\bfzeros\\
				\mbf{P}^\transpose\bshz_S
		\end{bmatrix}}_2\nonumber \\
		& \leq \frac{\alpha}{m^*} \|\mbf{P}^\transpose\bshz_S\|_2\leq \frac{\alpha}{m^*} \sum_{j \in S}\|\what{\mbf{z}}_j\|_2\leq \alpha. 
\end{align*}
where $\alpha\leq 1$ is given in \eqref{eq: mutual incoherence} and we used the fact that $\|\what{\mbf{z}}_j\|_2\leq 1$ (see Proposition \ref{prop: KKT conditions}), for $j \in S$, and $\|\mbf{P}^\transpose\|_2\leq 1$. As a result, from \eqref{eq: zsc2} and the preceding inequality, we have
\begin{align}
	\max_{j \in S^c}\|\bshz_j\|_2&\leq \alpha  + \max_{j \in S^c}\norm{\bsJ_{j}^\transpose\boldsymbol{\Gamma}_S\left(\frac{\mbf{v}}{\lambda_TT}\right)}_2=\alpha+\kappa_2. 
\end{align}
On the other hand, in light of Lemma \ref{lma: proof Theorem 4 kappa bound}, $\kappa_2 \leq 0.5(1-\alpha)$ with probability at least $1-2\exp(-\delta^2T/2)$, for $\delta >0$. 
Putting together the pieces, we have $\max_{j \in S^c}\|\bshz_j\|_2\leq 0.5(1+\alpha)<1$, thereby establishing the strict dual feasibility condition. 

{\em Part (c)}: From Assumption (A2) and Proposition \ref{prop: sub-matrix full rank}, we have
\begin{align}\label{eq: dagger identity}
    \bsPsi^\dagger_S\bsPsi_S=\mathrm{Blkdiag}(\mbf{I}_{n},\mbf{I}_{t_S},\mbf{\bsPsi}_{S,[d-1:0]}^\dagger\mbf{\bsPsi}_{S,[d-1:0]}),
\end{align}
where $t_S=(N-d+1)m^*$. Thus, we have 
\begin{align}\label{eq: delayed extraction}
\hspace{-2.0mm}{\mbf{u}}_{S,[0:N-d]}^*\!-\!\what{\mbf{u}}_{S,[0:N-d]}&=\boldsymbol{\Pi}_{S,[0:N-d]}{\bsPsi_S^\dagger\bsPsi_S}(\bsteta^*_S-\bshteta_S),
\end{align}
where $\boldsymbol{\Pi}_{S,[0:N-d]}=[\mbf{0}_{t_S\times n} \,\mbf{I}_{t_S\times t_S}\,\mbf{0}_{t_S\times dm^*}]$ and 
\begin{align}
\bsteta^*_S\!-\!\bshteta_S\!=\!\begin{bmatrix}
    \mbf{x}^*_0-\what{\mbf{x}}_0\\
    {\mbf{u}}_{S,[0:N-d]}^*-\what{\mbf{u}}_{S,[0:N-d]}\\
    {\mbf{u}}_{S,[N-d+1:0]}^*-\what{\mbf{u}}_{S,[N-d+1:0]}
\end{bmatrix}. 
\end{align}
For brevity, let $\boldsymbol{\Pi}=\boldsymbol{\Pi}_{S,[0:N-d]}$. 
From \eqref{eq: delayed extraction} and \eqref{eq: difference theta2}, and the sub-multiplicative property of norms, we have  
\begin{align}\label{eq: min_signal bound proof}
\|{\mbf{u}}_{S,[0:N-d]}^*\!-\!\what{\mbf{u}}_{S,[0:N-d]}\|_{\infty}
	&\leq \|\boldsymbol{\Pi}\bsPsi_S^\dagger \mbf{v}\|_\infty \nonumber \\
	&+\lambda_T\norm{\boldsymbol{\Pi}(\bsPsi_S^\transpose\bsPsi_S/T)^\dagger}_\infty, 
\end{align}
 where we used the fact $\|\widehat{\mbf{z}}_S\|_\infty\leq 1$. The second term is deterministic. Instead, the first term is random, and, from Lemma 3, it is upper bounded by $ \sigma/\sqrt{c_\mathrm{min}}(\sqrt{2\log(t_S)/T}+\delta)$ with probability at least $1-2\exp(-T\delta^2/2)$. Finally, the left-hand side of \eqref{eq: min_signal bound proof} can be written as $\max_{j \in S}\|{\mbf{u}}_{j,[0:N-d]}^*\!-\!\what{\mbf{u}}_{j,[0:N-d]}\|_\infty$. Putting the pieces together, we have the inequality in \eqref{eq: beta min}.

\smallskip 
{\em Part (d)}: By the triangle inequality, for all $j \in S$, we have 
\begin{align*}
    \|\mbf{u}_{j:[0:N-d]}^*\|_\infty&=\|\mbf{u}_{j:[0:N-d]}^*-\what{\mbf{u}}_{j:[0:N-d]}+\what{\mbf{u}}_{j:[0:N-d]}\|_\infty\\
    &\leq\|\mbf{u}_{j:[0:N-d]}^*-\what{\mbf{u}}_{j:[0:N-d]}\|_\infty+\|\what{\mbf{u}}_{j:[0:N-d]}\|_\infty\\
    &\overset{(i)}{\leq} \bsbeta_\mathrm{min}+\|\what{\mbf{u}}_{j:[0:N-d]}\|_\infty,
\end{align*}
where (i) follows from part (c). Thus, $\|\what{\mbf{u}}_{j:[0:N-d]}\|_\infty> 0$ if $\|\mbf{u}_{j:[0:N-d]}^*\|_\infty>\bsbeta_\mathrm{min}$. This observation together with $\what{S}\subseteq S$ in part (a) implies that $\what{S}=S$. 

Finally, the probability stated in the theorem is obtained by taking the union bound of the event where the dual feasibility holds and the event where $\ell_\infty$ bounds hold. \QEDB

\medskip 
\textit{Proof of Theorem \eqref{thm: frequency MIC}}: Consider the auxiliary system $\mbf{x}[k+1]=\mbf{A}\mbf{x}[k]+\mbf{b}_j{u}^*_j[k]$, where $j \in S^c$ and ${u}^*_j[k]=0$, $k\geq N$. Let $\mbf{x}[0]=0$. Thus, $\mbf{y}=\mbf{J}_j\mbf{u}^*_j$, where $\mbf{J}_j$ is given by \eqref{eq: system matrices} and $\mbf{y}$ and $\mbf{u}^*_j$ as in \eqref{eq: time collected vectors}. Let $\bsPsi_S$ be as in \eqref{eq: measurement model1*}, and consider
\begin{align}\label{eq: deadbeat filter0}
   \hspace{-3.0mm} \begin{bmatrix}
        \wtilde{\mbf{y}}[0]\\
        \vdots \\
        \wtilde{\mbf{y}}[N]
    \end{bmatrix}\triangleq \bsPsi_S^\dagger\begin{bmatrix}
        \mbf{y}[0]\\
        \vdots \\
        \mbf{y}[N]
    \end{bmatrix}=\bsPsi_S^\dagger\mbf{J}_j\mbf{u}^*_j=\bsPsi_S^\dagger\mbf{J}_j\begin{bmatrix}
        {u}^*_j[0]\\
        \vdots \\
        {u}^*_j[N]
    \end{bmatrix}, 
\end{align}

By assumption we have $\mathrm{nRank} \mathcal{Z}_S=n+m^*$. Thus, for all
$z\not \in \text{spec}(A)$, $\mathcal{G}_S[z]$ has full column rank and $\mathcal{G}_S^{+}[z]=[\mathcal{G}_S[z]^\transpose \mathcal{G}_S[z]]^{-1}\mathcal{G}_S[z]^\transpose$ and $\mathcal{G}_S^{+}[z]\mathcal{G}_S[z]=z^{-d}\mbf{I}$; see \cite[Theorem 1]{SK-HP-EZ-DSB:11}. Let $\wtilde{\mbf{y}}[z]$ be the $\mathcal{Z}$-transform of $\{\wtilde{\mbf{y}}[k]\}_{k=0}^\infty$.
 Then by using the construction given in \cite[pp. 49-50]{AnsariPhd2018} and the uniqueness of pseudo inverse \cite{AA-DSB:19}, we have $\wtilde{\mbf{y}}[z]=z^{-d}\mathcal{H}[z]{u}_j^*[z]$, where $\mathcal{H}_j[z]=\mathcal{G}_S^{+}[z]\mathcal{G}_j[z]$, for all $z \not\in \text{spec}(A)$.
 
 From Parsevel's theorem, we have the following bound: 
\begin{align}\label{eq: parsvel bound}
    \sqrt{\sum_{k=0}^{\infty}\|\wtilde{\mbf{y}}[k]\|^2_2}&=\sqrt{\frac{1}{2\pi}\int_{-\pi}^\pi \|\wtilde{\mbf{y}}[e^{j\omega}]\|_2^2d\omega}\nonumber\\
    &=\sqrt{\frac{1}{2\pi}\int_{-\pi}^\pi \|e^{-dj\omega}\mathcal{H}_j[e^{j\omega}]{u}_j^*[e^{j\omega}]\|_2^2d\omega}\nonumber\\
    &\leq \sup_{\{\omega\in [-\pi, \pi ]\}}\|\mathcal{H}_j[e^{j\omega}]\|_2 \sqrt{\frac{1}{2\pi}\int_{-\pi}^\pi|u^*_j[e^{j\omega}]|^2_2d\omega}\nonumber\\
    &= \sup_{\{z\in \mathbb{C}:|z|=1\}}\|\mathcal{H}_j[z]\|_2\|\mbf{u}_j^*\|_2.
\end{align}
For the final inequality, we once again used Parsevel's theorem and the fact that $u^*[k]=0$, for all $k > N$. 

On the other hand, from \eqref{eq: deadbeat filter0} and \eqref{eq: parsvel bound}, we have 
\begin{align}
\|\bsPsi_S^\dagger\mbf{J}_j\|_2&=\sup_{\|\mbf{u}^*_j\|_2=1} \|\bsPsi_S^\dagger\mbf{J}_j\mbf{u}^*_j\|_2=\sup_{\|\mbf{u}^*_j\|_2=1}\sqrt{\sum_{k=0}^{N}\|\wtilde{\mbf{y}}[k]\|^2_2}\nonumber \\
&\leq \sup_{\|\mbf{u}^*_j\|_2=1}\sqrt{\sum_{k=0}^{\infty}\|\wtilde{\mbf{y}}[k]\|^2_2}\leq \sup_{\{z\in \mathbb{C}:|z|=1\}}\|\mathcal{H}_j[z]\|_2.\nonumber 
\end{align}
Thus $\max_{\{j \in S^c\}}\sup_{\{z\in \mathbb{C}:|z|=1\}}\|\mathcal{H}[z]\|_2\leq \alpha/m^*$ implies that  $\|\bsPsi_S^\dagger\mbf{J}_j\|_2\leq \alpha/m^*$. The proof is now complete. \QEDB  

\medskip 
\textit{Proof of Theorem \eqref{thm: l2-consistency}}:    From Theorem \ref{thm: location-consistency}, $S=\what{S}$ holds with probability at least $1-4\exp(-T\delta^2/2)$. Thus, from \eqref{eq: OLS estimate}, with the same probability, we have
\begin{align}\label{eq: least and lasso set equal}
    \what{\bsbeta}^{(OLS)}_{\what{S},[0:N-d]}=\what{\bsbeta}^{(OLS)}_{{S},[0:N-d]}=   \widetilde{\boldsymbol{\Pi}}_{{S},[0:N-d]}\bsPsi^\dagger \mbf{y}, 
\end{align}
   where $\widetilde{\boldsymbol{\Pi}}_S\triangleq \widetilde{\boldsymbol{\Pi}}_{{S},[0:N-d]}=[\mbf{I}_{n+t_S} \, \mbf{0}_{(n+t_S)\times dm^*}]$.
   
   Since $\mbf{y}\sim \mathcal{N}(\bsbeta^*_{S,[0:N-d]},\sigma^2\mbf{I})$, it follows that 
   \begin{align}
       \what{\bsbeta}^{(OLS)}_{S,[0:N-d]}-\bsbeta^*_{S,[0:N-d]}\sim \mathcal{N}(\mbf{0},\underbrace{\sigma^2\widetilde{\boldsymbol{\Pi}}_S(\bsPsi_{{S}}^\transpose\bsPsi_{{S}})^\dagger \widetilde{\boldsymbol{\Pi}}^\transpose_S}_{\triangleq \boldsymbol{\Sigma}\,\in\, \mathbb{R}^{n+t_S\times n+t_S}}). 
   \end{align}
We now upper bound $\|\boldsymbol{\Sigma}\|_2$. Recall from \eqref{eq: Psi} and Proposition \ref{prop: sub-matrix full rank} (ii) that $\bsPsi_{{S}}=[\bsPsi_{{S},[N:d]}\,\bsPsi_{{S},[d-1:0]}]$ and $\mathcal{R}(\bsPsi_{{S},[N:d]})\cap \mathcal{R}(\bsPsi_{{S},[d-1:0]})=\{0\}$. Let $\mbf{M}=[\mbf{I}-\bsPsi_{{S},[d-1:0]}\bsPsi_{{S},[d-1:0]}^\dagger]$. Then, by invoking \cite[Lemma D]{AA-DSB:19}, we have
\begin{align}
    \widetilde{\boldsymbol{\Pi}}_S(\bsPsi_{{S}}^\transpose\bsPsi_{{S}})^\dagger \widetilde{\boldsymbol{\Pi}}^\transpose_S=[(\mbf{M}\bsPsi_{{S},[N:d]})^\dagger(\mbf{M}\bsPsi_{{S},[N:d]})]^\transpose. 
\end{align}
Since $\mbf{M}^*=\mbf{M}=\mbf{M}^2$, from Assumption (A2), it follows that $\norm{\boldsymbol{\Sigma}}_2\leq \sigma^2/(Tc_\text{min})$. 

From the second concentration result in Lemma \ref{lma: sub-gaussian norm}, we have 
  \begin{align}\label{eq: norm bound least squares}
\norm{\bsbeta^*_{S,[0:N-d]}-\what{\bsbeta}^{(OLS)}_{{S},[0:N-d]}}_2&\!\leq\!\frac{4\sigma}{\sqrt{c_{min}}}\left\lbrace\sqrt{\frac{(n+t_S)}{T}}\right\rbrace\nonumber \\
			& \hspace{-8.0mm} +\frac{2\sigma}{\sqrt{c_{min}}}\left\lbrace\sqrt{\frac{1}{T}\log\left(\frac{1}{{\delta_1}}\right)}\right\rbrace,
  \end{align}
  with probability at least $1-\delta_1$ for $\delta_1 (0,1)$. The statement of the theorem follows by taking an union bound over the events where \eqref{eq: norm bound least squares} and \eqref{eq: least and lasso set equal} hold. \QEDB
  
  \begin{lemma}\label{lma: proof Theorem 4 kappa bound}  With the notation and assumptions stated in Theorem \ref{thm: location-consistency}, we have $\mathbb{P}[\max_{j \in S^c}\norm{\bsJ_{j}^\transpose\boldsymbol{\Gamma}_S\left({\mbf{v}}/{\lambda_TT}\right)}_2\!\geq\! 0.5(1-\alpha)]\leq 2\exp(-T\delta^2/2)$, where $\boldsymbol{\Gamma}_S\!=\![\mbf{I}\!-\!\bsPsi_S\bsPsi_S^\dagger]$ and $\alpha \in [0,1)$.
  \end{lemma}
  \begin{proof}
  Let $\widetilde{\alpha}=0.5(1-\alpha)$ and take the union bound to get 
\begin{align}\label{eq: union bound one}
    \mathbb{P}\left[\max_{j \in S^c}\norm{\bsJ_{j}^\transpose\boldsymbol{\Gamma}_S\left(\frac{\mbf{v}}{\lambda_TT}\right)}_2\!\geq\!\widetilde{\alpha}\right]\nonumber\\
& \hspace{-21.0mm}\leq \sum_{j \in S^c}\mathbb{P}\left[\norm{\bsJ_{j}^\transpose\boldsymbol{\Gamma}_S\left(\frac{\mbf{v}}{\lambda_TT}\right)}_2\!\geq\!\widetilde{\alpha}\right]. 
\end{align}
  Since $\mbf{v}\sim \mathcal{N}(\mbf{0},\sigma^2\mbf{I})$, we have $\bsJ_{j}^\transpose\boldsymbol{\Gamma}_S\left({\mbf{v}}/{\lambda_TT}\right)\sim \mathcal{N}(\mbf{0},\boldsymbol{\Sigma}_j)$ with $\boldsymbol{\Sigma}_j=\bsJ_{j}^\transpose\boldsymbol{\Gamma}_S\boldsymbol{\Gamma}_S^\transpose\bsJ_{j}/(\lambda^2_TT^2)$. Furthermore, from the identity that $\|\mbf{X}\mbf{X}^\transpose\|_2=\|\mbf{X}\|_2^2$, the following is trivial. 
  \begin{align}
      \|\boldsymbol{\Sigma}_j\|_2=\frac{1}{\lambda_T^2T^2}\|\bsJ_{j}^\transpose\boldsymbol{\Gamma}_S\|_2^2\leq \frac{1}{\lambda_T^2T^2}\|\bsJ_{j}^\transpose\|_2^2\leq  \frac{C^2T}{\lambda_T^2T^2}. 
  \end{align}
  The first inequality follows because $\boldsymbol{\Gamma}_S$ is a projection matrix and for the last inequality from the normalization assumption (A1). Invoking Lemma \ref{lma: sub-gaussian norm}, we bound the inequality in \eqref{eq: union bound one} as 
  \begin{align*}
      \sum_{j \in S^c}\mathbb{P}\left[\norm{\bsJ_{j}^\transpose\boldsymbol{\Gamma}_S\left(\frac{\mbf{v}}{\lambda_TT}\right)}_2\!\geq\!\widetilde{\alpha}\right]\leq \sum_{j \in S^c}c_N\exp\left(-\frac{\wtilde{\alpha}^2\lambda_T^2T}{8\sigma^2C^2}\right),
  \end{align*}
  where $c_N=5^{N+1}$. Since each term in the summand is the same, the right-hand side can be expressed as
  \begin{align}\label{eq: exp_bound}
      \exp\left((N+1)\log(5)+\log(m-m_0)-\frac{\wtilde{\alpha}^2\lambda_T^2T}{8\sigma^2C^2}\right)
  \end{align}
  Substituting $\lambda_T$ (see \eqref{eq: optimal lambda}) and $\wtilde{\alpha}=0.5(1-\alpha)$ in \eqref{eq: exp_bound}, and simplifying it gives us the required bound.   
  \end{proof}
  
    \begin{lemma}\label{lma: min beta probabilitic bound} With the notation and assumptions stated in Theorem \ref{thm: location-consistency}, for $\delta \in [0,1)$, we have $\mathbb{P}[\|\boldsymbol{\Pi}_{S,[0:N-d]}\bsPsi_S^\dagger \mbf{v}\|_\infty\geq \sigma/\sqrt{c_\mathrm{min}}(\sqrt{2\log(t_S)/T}+\delta)]\leq 2\exp(-T\delta^2/2)$. 
  \end{lemma}
  \begin{proof}
Recall that $\boldsymbol{\Pi}_{S,[0:N-d]}=[\mbf{0}_{t_S\times n} \,\mbf{I}_{t_S\times t_S}\,\mbf{0}_{t_{S}\times dm^*}]$ and $t_{S}=(N-d+1)m^*$. We drop the sub-script notation $[0:N-d]$ in $\boldsymbol{\Pi}_{S,[0:N-d]}$. Let $z_l=\mbf{e}_l^\transpose\boldsymbol{\Pi}_{S}\bsPsi_S^\dagger \mbf{v}$ be the $l^{th}$ entry of the vector $\boldsymbol{\Pi}_{S}\bsPsi_S^\dagger \mbf{v}$, where $\mbf{e}_l$ is the $l^{th}$ canonical basis vector in $ \mathbb{R}^{t_S}$. Thus $\|\boldsymbol{\Pi}_{S,[0:N-d]}\bsPsi_S^\dagger \mbf{v}\|_\infty=\max_{l \in 1\ldots t_S} |z_l|$ and for $\kappa \geq 0$ by invoking union bound we have 
\begin{align}\label{eq: union bound}
    \mathbb{P}[\max_{l \in 1\ldots t_S} |z_l|\geq \kappa]\leq \sum_{l=1}^{t_S} \mathbb{P}[|z_l|\geq \kappa]. 
\end{align}

We bound terms on the right-hand side by invoking standard concentration results. We first characterize the distribution of $z_l$. Since $\mbf{v}\sim \mathcal{N}(\mbf{0},\sigma^2\mbf{I})$, we have $z_l\sim \mathcal{N}(\mbf{0},{\sigma}^2_l)$, where
\begin{align}\label{eq: sigma_l}
    \sigma^2_l=\sigma^2\mbf{e}_l^\transpose\boldsymbol{\Pi}_{S}\bsPsi_S^\dagger(\boldsymbol{\Pi}_{S}\bsPsi_S^\dagger)^\transpose\mbf{e}_l&\leq \sigma^2\lambda_\text{max}(\boldsymbol{\Pi}_{S}\bsPsi_S^\dagger(\boldsymbol{\Pi}_{S}\bsPsi_S^\dagger)^\transpose)\nonumber\\
    &=\sigma^2\|\boldsymbol{\Pi}_{S}(\bsPsi_{{S}}^\transpose\bsPsi_{{S}})^\dagger\boldsymbol{\Pi}_{S}^\transpose\|_2\nonumber \\
    &\leq \sigma^2\|\widetilde{\boldsymbol{\Pi}}_S(\bsPsi_{{S}}^\transpose\bsPsi_{{S}})^\dagger\widetilde{\boldsymbol{\Pi}}_S^\transpose\|_2\nonumber \\
    & \leq \sigma^2/(Tc_\text{min}). 
\end{align}
where $\widetilde{\boldsymbol{\Pi}}_S=[\mbf{I}_{n+t_S} \, \mbf{0}_{(n+t_S)\times dm^*}]$. The second inequality follows from interlacing property of singular values. The final inequality is showed in the proof of Theorem \ref{thm: l2-consistency}. 

Since $z_l$ is Gaussian, from \cite[page 22]{MJW:19} and \eqref{eq: sigma_l}, we have $\mathbb{P}[|z_l|\geq \kappa]\leq \exp(-\kappa^2/(2\sigma^2_l))\leq \exp(-\kappa^2Tc_\text{min}/(2\sigma^2))$. Substituting this inequality in \eqref{eq: union bound}, we find that 
\begin{align}\label{eq: union bound}
    \mathbb{P}[\max_{l \in 1\ldots t_S} |z_l|\geq \kappa]\leq \exp\left(\log(t_S)-\frac{\kappa^2Tc_\mathrm{min}}{2\sigma^2}\right). 
\end{align}
The result follows by letting $\kappa=\sigma/\sqrt{c_\mathrm{min}}(\sqrt{2\log(t_S)/T}+\delta)$ and simplifying terms in the exponential term. 
  \end{proof}

\begin{lemma}\label{lma: sub-gaussian norm}
    Let $\mbf{p}\sim \mathcal{N}(\mbf{0},\boldsymbol{\Sigma})$, where $\boldsymbol{\Sigma}\in \mathbb{R}^{l\times l}$ is a positive definite matrix. Then, $\mathbb{P}[\|\mbf{p}\|_2\geq t]\leq 5^l\exp(-t^2/(8\|\boldsymbol{\Sigma}\|_2))$. Furthermore, $\|\mbf{p}\|_2\leq 4\sqrt{\|\boldsymbol{\Sigma}\|_2l}+2\sqrt{\|\boldsymbol{\Sigma}\|_2\log(1/\delta)}$ with 
     probability at least $1-\delta$ for $\delta (0,1)$.  
\end{lemma}
\begin{proof}
    Follows from Lemma 8.2 and Theorem 8.3 in \cite{ARlecnotes2019}. 
\end{proof}
\end{document}